\documentclass[a4paper,11pt,reqno]{amsart}
\usepackage{amsmath}
\usepackage{amssymb}
\usepackage{amsthm}
\usepackage[english]{babel}
\usepackage{hyperref}
\usepackage{mathrsfs}
\usepackage{tikz}

\usepackage{array,graphicx}

\usepackage[export]{adjustbox}
\usepackage{tikz}
\usetikzlibrary{fit,tikzmark}   
\usepackage{booktabs,cellspace}

\usepackage{enumitem}

\numberwithin{equation}{section}

\newtheorem{theor}[equation]{Theorem}
\newtheorem{lemma}[equation]{Lemma}

\newtheorem{defi}[equation]{Definition}

\newcommand{\cont}{\mathrm{cont}}

\newcommand{\res}{\mathrm{res}}
\newcommand{\Ind}{\mathrm{Ind}}
\newcommand{\Res}{\mathrm{Res}}

\newcommand{\s}{{\sf S}}

\newcommand{\Md}{\!\mod}

\newcommand{\Z}{\mathbb{Z}}

\renewcommand{\epsilon}{\varepsilon}
\renewcommand{\phi}{\varphi}
\newcommand{\xymat}{\xymatrix@R=6pt@C=10pt}

\newcommand{\la}{\lambda}
\newcommand{\be}{\beta}
\newcommand{\al}{\alpha}
\newcommand{\eps}{\epsilon}
\newcommand{\ga}{\gamma}
\newcommand{\de}{\delta}

\def\Par{{\mathscr {P}}}

\def\Tsym{T^{\mathrm{sym}}}
\def\Tspin{T^{\mathrm{spin}}}

\newcommand{\dbl}{\mathbf{\mathrm{dbl}}}

\def\Md#1{\text{ }(\text{\rm mod } #1)\,}

\begin{document}

\title[Decomposition numbers of 2-parts spin representations]{Decomposition numbers of 2-parts spin representations of symmetric groups in characteristic 2}

\author{\sc Lucia Morotti}
\address{Leibniz Universit\"{a}t Hannover\\ Institut f\"{u}r Algebra, Zahlentheorie und Diskrete Mathematik\\ 30167 Hannover\\ Germany} 
\address{Mathematisches Institut\\ Heinrich-Heine-Universit\"{a}t D\"{u}sseldorf\\ 40225 D\"{u}sseldorf\\ Germany} 
\email{lucia.morotti@uni-duesseldorf.de}

\thanks{During part of the work the author was supported by the DFG grant MO 3377/1-2. While working on the revised version the author was working at the Department of Mathematics of the University of York, supported by the Royal Society grant URF$\backslash$R$\backslash$221047.}

\begin{abstract}
We give explicit formulas to compute most of the decomposition numbers of reductions modulo 2 of irreducible spin representations of symmetric groups indexed by partitions with at most 2 parts. In many of the still open cases small upper bounds are found.
\end{abstract}

\maketitle

\section{Introduction}

Let $D$ be an irreducible representation of a double cover $\widetilde{\s}_n$ of a symmetric group $\s_n$. We say that $D$ is a spin representation if $D$ cannot be viewed also as a representation of $\s_n$.

It is well known that in characteristic 0 (pairs of) irreducible spin representations of the symmetric groups are labeled by strict partitions, that is partitions in distinct parts, see \cite{schur,s}. Not much is known about decomposition matrices of spin representations of symmetric groups. For example in general not even the shape of the decomposition matrix is known.

When reducing characteristic 0 spin representations modulo an odd prime, the obtained representations are still spin representations. In this case, which will not be considered in this paper, results on decomposition numbers consider maximal composition factors (that is, under a specified ordering of the columns, the last non-zero entry in each row of the decomposition matrix) \cite{BK2,BK3}, the shape of the decomposition matrix for small primes \cite{abo1,bmo} and decomposition numbers in certain specific blocks or classes of modules or for small $\widetilde{\s}_n$ \cite{fkm,ma,m2,Mu,Wales,y}.

On the other hand reductions modulo 2 of spin representations may also be viewed as representations of symmetric groups. In this case maximal composition factors and their multiplicities have been found in \cite{ben,bo}. This result can be used to rule out some characteristic 2 modules as been composition factors of a given spin representation. An improvement in this direction has been obtained in \cite[Lemma 4.2]{m}. Apart for the small $n$ cases \cite{gap,ModularAtlas}, the only other classes of modules for which decomposition numbers are known in this case are basic and second basic spin representations \cite{Wales} or RoCK blocks \cite[Section 5]{f2} and \cite[Section 5]{f3}.


One particular class of modules of symmetric groups for which decomposition numbers are known are Specht modules indexed by partitions with at most 2 parts. In this case decomposition numbers have been found by James in \cite{j1,j2} (see also \cite[Theorem 24.15]{JamesBook}). The corresponding question, studying composition factors of reductions modulo $p$ of spin representations labeled by partitions with at most to parts, has been studied in \cite{m2} in odd characteristic. There irreducible characteristic $p$ spin representations which are composition factors of some (though not one particular) such characteristic 0 spin representation were explicitly described. Further it was shown that the corresponding part of the decomposition matrix is block triangular (with blocks corresponding to representations indexed by the same partition).

In this paper we will consider the above problem in characteristic 2, describing modules which are composition factors of the reduction modulo 2 of some spin representation with at most 2 parts and finding formulas for computing most of the corresponding decomposition numbers.

For any 2-regular partition $\la\vdash n$ let $D^\la$ be the corresponding characteristic 2 irreducible representation of the symmetric group $\s_n$ and $S(\la,\eps)$ be the corresponding characteristic 0 irreducible spin representation(s) of the double cover $\widetilde{\s}_n$,  with $\eps=0$ or $\pm$ depending on $\la$.

The first result we obtain is the following (see Section 2 for the definition of the double of a partition):

\begin{theor}\label{c1}
If $0\leq a\leq\lfloor(n-1)/2\rfloor$ and $\mu\in\Par_2(n)$ is such that $D^\mu$ is a composition factor of $S((n-a,a),\eps)$ then $\mu$ has at most 2 parts or it is the double of a partition with at most 2 parts.
\end{theor}

The above result leads us to study decomposition numbers of the forms $[S((n-a,a),\eps):D^{(n-b,b)}]$ and $[S((n-a,a),\eps):D^{\dbl(n-b,b)}]$. In the first case, provided $b<(n-1)/2$, we will give exact formulas for decomposition numbers in Theorem \ref{t1}. This result shows any module of the form $D^{(n-b,b)}$ is indeed a composition factor of the reduction modulo 2 of some module of the form $S((n-a,a),\eps)$. Further Theorem \ref{t1} shows that at most 2 rows of the corresponding part of the decomposition matrix are non-zero. In the second case it is known by \cite{ben} that $D^{\dbl(n-b,b)}$ is a composition factor of $S((n-b,b),\eps)$ and that the corresponding part of the decomposition matrix is triangular. We will compute most of the corresponding decomposition numbers in Theorems \ref{T100822} and \ref{T170223} and find some upper bounds in many of the other cases. In particular we find formulas or upper bounds for all but one column of the corresponding part of the decomposition matrix.

Before being able to state Theorem \ref{t1}, \ref{T100822} and \ref{T170223} we need some definitions.

\begin{defi}
Given $m\geq 0$, if $m=2^{a_1}+\ldots+2^{a_k}$ with $a_1>\ldots>a_k\geq 0$, let $d_m:=a_1+k-3$.
\end{defi}

As in \cite[Definition 24.12]{JamesBook}, for integers $\ell$ and $m$ with $\ell\geq 0$ we say that $\ell$ contains $m$ to base $2$ if there exists $k$ with  $0\leq m<2^k\leq\ell$ and further, for $\ell=\sum_ia_i2^i$ and $m=\sum_ib_i2^i$ the $2$-adic decompositions of $\ell$ and $m$, $b_i\in\{0,a_i\}$ for each $i$.

\begin{defi}
For $\ell,m$ with $\ell\geq 0$ let $g_{\ell,m}:=1$ if $\ell$ contains $m$ to base 2 or $g_{\ell,m}:=0$ else.
\end{defi}

The next theorem gives exact formulas for the decomposition numbers of the form $[S((n-a,a),\eps):D^{(n-b,b)}]$ with $b<(n-1)/2$.

\begin{theor}\label{t1}
Let $p=2$, $0\leq a\leq\lfloor(n-1)/2\rfloor$ and $0\leq b\leq\lfloor(n-3)/2\rfloor$. Then $[S((n-a,a),\eps):D^{(n-b,b)}]=2^{d_{n-2b+1}}$ if one of the following holds:
\begin{itemize}
\item $n\equiv 0\Md{2}$ and $n-2a=2$,

\item $n\equiv 0\Md{8}$, $2\mid b$ and $n-2a=4$,

\item $n\equiv 4\Md{8}$, $2\nmid b$ and $n-2a=4$,

\item $n\equiv 1\text{ or }3\Md{8}$, $2\mid b$ and $n-2a=1$,

\item $n\equiv 1\text{ or }3\Md{8}$, $2\nmid b$ and $n-2a=3$,

\item $n\equiv 5\text{ or }7\Md{8}$, $2\nmid b$ and $n-2a=1$,

\item $n\equiv 5\text{ or }7\Md{8}$, $2\mid b$ and $n-2a=3$.
\end{itemize}
In all other cases $[S((n-a,a),\eps):D^{(n-b,b)}]=0$.
\end{theor}

In the next theorem we describe most decomposition numbers of the form $[S((n-a,a),\eps):D^{\dbl(n-b,b)}]$.

\begin{theor}\label{T100822}
Let $p=2$, $0\leq a\leq\lfloor(n-1)/2\rfloor$ and $1\leq b<n/2$ with $\dbl(n-b,b)\in\Par_2(n)$. Then
\begin{itemize}
\item $[S((n-a,a),\eps):D^{\dbl(n-b,b)}]=g_{n-2b,a-b}$ if one of the following holds:
\begin{itemize}
\item $n\equiv 1\Md{4}$ and $b$ is even,

\item $n\equiv 2\Md{4}$ and $b$ is odd,

\item $n\equiv 3\Md{4}$ and $b$ is odd,
\end{itemize}

\item $[S((n-a,a),\eps):D^{\dbl(n-b,b)}]=2g_{n-2b,a-b}$ if $n\equiv 2\Md{4}$ and $b$ is even,

\item $[S((n-a,a),\eps):D^{\dbl(n-b,b)}]=g_{n-2b,a-b}-g_{n-2b-2,a-b-1}$ if one of the following holds:
\begin{itemize}
\item $n\equiv 1\Md{4}$ and $b$ is odd,

\item $n\equiv 3\Md{4}$ and $b$ is even,
\end{itemize}

\item $[S((n-a,a),\eps):D^{\dbl(n-b,b)}]=g_{n-2b,a-b}+2g_{n-2b-2,a-b-1}$ if $n\equiv 0\Md{4}$ and $b$ is odd.
\end{itemize}
\end{theor}

The cases $b=0$ or $n\equiv 0\Md{4}$ and $b$ even are not covered by Theorem \ref{T100822}. In the second case, for $b\geq 2$, though we are not able to compute all decomposition numbers exactly, we can still find upper bounds and in some cases exact decomposition numbers. In the next theorem $\nu_2$ is the 2-adic valuation.

\begin{theor}\label{T170223}
Let $p=2$, $n\equiv 0\Md{4}$, $2\leq b\leq (n-6)/2$ even and  $0\leq a\leq (n-2)/2$. Then
\[[S((n-a,a),\eps):D^{\dbl(n-b,b)}]\leq 2g_{n-2b-3,a-b}+2g_{n-2b-3,a-b-3}
\]
with equality holding if
\begin{align*}
&g_{n-2b+1,c-b}+g_{n-2b-1,c-b-1}-g_{n-2b-3,c-b-2}\\
&=\left\{\begin{array}{ll}
\de_{c\not=n/2}g_{n-2b-4,c-b}+(1+\de_{c\not=n/2})g_{n-2b-4,c-b-4},&c\text{ is even,}\\
2g_{n-2b-4,c-b-1}+g_{n-2b-4,c-b-5},&c\text{ is odd.}
\end{array}\right.
\end{align*}
holds for some $c\in\{a,a+1\}$ with $c-b\equiv 0$ or $1\Md{4}$.

In particular if $a-b\equiv 2\Md{4}$ then $[S((n-a,a),\eps):D^{\dbl(n-b,b)}]=0$. If $a-b\not\equiv 2\Md{4}$ then equality holds if $\nu_2(\lfloor(a-b+1)/4\rfloor)\geq\nu_2((n-2b)/4)$ or $g_{n-2b-4,4\lfloor (a-b+1)/4\rfloor-4}=0$.
\end{theor}

Note that if $b\equiv n/2-2\Md{4}$, then $\nu_2(\lfloor(a-b+1)/4\rfloor)\geq\nu_2((n-2b)/4)$ always holds. In particular about half of the columns covered in the above theorem can be completely computed through it.

Further, by \cite[Theorem 1.4]{m3},
\[[S((n-a,a),\eps):D^{\dbl(n-b,b)}]=[S((n-a+2,a+2),\eps):D^{\dbl(n-b+2,b+2)}]\]
whenever $a,b>0$ and all of the above modules are defined (in many cases this follows also from Theorems \ref{T100822} and \ref{T170223}). In particular decomposition numbers $[S((n-a,a),\eps):D^{\dbl(n-b,b)}]$ with $n\equiv 0\Md{4}$ and $b\geq 2$ even (that is those covered in Theorem \ref{T170223}) only depend on $n-2b$ and $a-b$.

The assumption $a>0$ in the previous paragraph could be dropped (using slightly more complicated formulas), but not the assumption $b>0$. For example, from decomposition matrices in GAP it can be recovered that $[S((7,1),0):D^{(5,3)}]=1$ but $[S((9,3),0):D^{(6,4,2)}]=2$ and that $[S((5,4),\pm):D^{(5,4)}]=1$ but $[S((7,6),\pm):D^{(6,5,2)}]=0$.

In Section \ref{s2} we will recall some basic definitions and results and prove Theorem \ref{c1}. In Section \ref{s3} we will prove some results on projective modules. Theorems \ref{t1}, \ref{T100822} and \ref{T170223} will then be proved in Sections \ref{s4}, \ref{s5} and \ref{s6} respectively. Some (partial) decomposition matrices, computed using the above results, are given in Appendix \ref{appendix}. 

Looking at Theorems \ref{c1} and \ref{t1} one may ask whether all composition factors of $S(\la,\eps)$ are of the form $D^{\dbl(\mu)}$ for some partition $\mu$ for all strict partitions $\la$ with $\dbl(\la)$ 2-regular. This is in general false. For example, looking at known decomposition matrices and comparing characters, it can be checked that
\begin{align*}
[S((10,5,1),\pm)]&=[D^{(7,6,3)}]+[D^{(6,4,3,2,1)}]
\end{align*}
and
\begin{align*}
[S((11,5,1),0)]&=8[D^{(17)}]+2[D^{(9,8)}]+4[D^{(9,7,1)}]+3[D^{(9,5,3)}]+2[D^{(8,6,3)}]\\
&\hspace{11pt}+2[D^{(7,6,4)}]+[D^{(7,6,3,1)}]+2[D^{(7,5,3,2)}]+[D^{(6,5,3,2,1)}].
\end{align*}

\section{Notation and basic results}\label{s2}

Let $n\geq 0$ and $\widetilde{\s}_n$ be a double cover of $\s_n$. Then there exists $z$ central in $\widetilde{\s}_n$ of order 2 with $\s_n\cong\widetilde{\s}_n/\langle z\rangle$. Representations of $\widetilde{\s}_n$ on which $z$ acts trivially can also be viewed as $\s_n$-representations, while those on which $z$ acts as $-1$ are called spin representations. Note that reductions modulo 2 of spin representations can always be viewed as representations of the corresponding symmetric group $\s_n$. In particular all their composition factors are irreducible characteristic 2 representations of $\s_n$ (viewed as $\widetilde{\s}_n$-representations).

Let $\Par(n)$ be the set of partitions of $n$. Further let $\Par_2(n)$ be the set of 2-regular partitions of $n$, that is partitions in distinct parts or strict partitions. For any partition $\la$, let $h(\la)$ be the number of parts of $\la$ and $h_2(\la)$ be the number of even parts of $\la$.

Identifying partitions and their Young diagrams, if $\la\in\Par(n)$ and $A=(i,j)$ is a node, we say that $A$ is a removable (resp. addable) node of $\la$ if $A\in\la$ (resp. $A\not\in\la$) and $\la\setminus \{A\}$ (resp. $\la\cup\{A\}$) is the Young diagram of a partition. If $\la\in\Par_2(n)$ we say that $A$ is a bar-removable (resp. bar-addable) node of $\la$ if $A$ is removable (resp. addable) and $\la\setminus\{A\}$ (resp. $\la\cup\{A\}$) is a strict partition.

It is well known, see for example \cite{JamesBook,JK,schur,s}, that $\Par(n)$ labels irreducible representations of $\s_n$ in characteristic 0, while $\Par_2(n)$ labels both the irreducible representations of $\s_n$ in characteristic 2 and (pairs) of irreducible spin representations. For $\la\in\Par(n)$ we denote by $S^\la$ the irreducible characteristic 0 representation of $\s_n$ labeled by $\la$. As in the introduction, for $\la\in\Par_2(n)$ define $D^\la$ to be the irreducible characteristic 2 representation and $S(\la,\eps)$ the irreducible spin representation(s) indexed by $\la$. Here $\eps=0$ if $n-h(\la)$ is even and $\eps\in\{\pm\}$ if $n-h(\la)$ is odd. In the following we will also work with modules $S(\la)$: for $\la\in\Par_2(n)$ we define $S(\la)$ to be either $S(\la,0)$ or $S(\la,+)\oplus S(\la,-)$ depending on the parity of $n-h(\la)$. Further, for any $\la\in\Par_2(n)$, let $P^\la$ be the indecomposable projective module of $\widetilde{\s_n}$ with socle $D^\la$.

For a partition $\la=(\la_1,\ldots,\la_h)$ with $h=h(\la)$, let
\begin{align*}
\dbl(\la)&=(\lceil (\la_1+1)/2\rceil,\lfloor (\la_1-1)/2\rfloor),\ldots,\lceil (\la_h+1)/2\rceil,\lfloor (\la_h-1)/2\rfloor)),\\
\overline{\dbl}(\la)&=(\lceil \la_1/2\rceil,\lfloor \la_1/2\rfloor),\ldots,\lceil \la_h/2\rceil,\lfloor \la_h/2\rfloor)).
\end{align*}
Further let $\la^R$ be the regularisation of $\la$ as defined in \cite[6.3.48]{JK} for $p=2$.

It is easy to check that $\overline{\dbl}(\la)$ is always a partition for any $\la\in\Par_2(n)$, so that in this case $(\overline{\dbl}(\la))^R$ is well defined.

Further if $\dbl(\la)\in\Par_2(n)$ then $\dbl(\la)=(\overline{\dbl}(\la))^R$. This can be checked by showing that $\dbl(a)=(\overline{\dbl}(a))^R$ for any $a\geq 1$. So
\[\dbl(\la)=((\overline{\dbl}(\la_1))^R,\ldots,(\overline{\dbl}(\la_h))^R)\]
and $\overline{\dbl}(\la)$ have the same number of nodes on each ladder. Since $\dbl(\la)$ is a 2-regular partition it follows that $\dbl(\la)=(\overline{\dbl}(\la))^R$.

The following lemma, which is an analog of James' regularisation result, has been proved in \cite[Theorem 1.2]{ben} and \cite[Theorem 5.1]{bo}.

\begin{lemma}\label{t2}
Let $\la,\mu\in\Par_2(n)$. If $[S(\la,\eps):D^\mu]>0$ then $\mu\unrhd(\overline{\dbl}(\la))^R$. Further $[S(\la,\eps):D^{(\overline{\dbl}(\la))^R}]=2^{\lfloor h_2(\la)/2\rfloor}$.
\end{lemma}

This result was improved in \cite[Lemma 4.2]{m} to obtain the following:

\begin{lemma}\label{t3}
Let $\la,\mu\in\Par_2(n)$. If $[S(\la,\eps):D^\mu]>0$ and $\mu$ is not the double of any partition then $h(\mu)\leq 2h(\la)-2$.
\end{lemma}

Theorem \ref{c1} then easily follows:

\begin{proof}[Proof of Theorem \ref{c1}]
If $[S((n-a,a),\eps):D^\mu]>0$ and $\mu$ is not the double of a partition then by Lemma \ref{t3} $h(\mu)\leq 2$. If instead $\mu=\dbl(\nu)$ then by Lemma \ref{t2} $\mu\unrhd(\overline{\dbl}((n-a,a)))^R$, so that
\[2h(\nu)-1\leq h(\mu)\leq h(\overline{\dbl}(n-a,a))\leq 4\]
and then $h(\nu)\leq 2$.
\end{proof}

Given any node $(i,j)$, let $\res(i,j):=j-i\Md{2}$ be the residue of $(i,j)$. Further define the bar-residue of $(i,j)$ to be $\overline{\res}(i,j)=0$ if $j\equiv 0$ or $3\Md{4}$ or $\overline{\res}(i,j)=1$ if $j\equiv 1$ or $2\Md{4}$. When considering $\res(i,j)$ we will in the following identify $\Z/2\Z$ with $\{0,1\}$ in the obvious way. For $\la$ any partition let the content of $\la$ be $\cont(\la):=(c_0,c_1)$ with $c_0$ (resp. $c_1$) the number of nodes of residue $0$ (resp. $1$) of $\la$. Similarly let the bar-content of $\la$ be $\overline{\cont}(\la):=(d_0,d_1)$ with $d_0$ (resp. $d_1$) the number of nodes of bar-residue $0$ (resp. $1$) of $\la$.

By \cite[2.7.41, 6.1.21 and 6.3.50]{JK} we have that $S^\la$ and $D^\mu$ are in the same block if and only if $\cont(\la)=\cont(\mu)$. Further by \cite[3.9 and 4.1]{bo} $S(\la,\eps)$ and $S^\mu$ are in the same block if and only if $\overline{\cont}(\la)=\cont(\mu)$.

For a given block $B$ we may thus define the content of $B$ as the content $\cont(\la)$ of any module $S^\la$ or $D^\la$ contained in $B$ or equivalently as the bar-content $\overline{\cont}(\la)$ of any module $S(\la,\eps)$ contained in $B$.

If $B$ is a block of $\widetilde{\s}_n$ with content $(c_0,c_1)$ and $V$ is any module contained in $B$, let $e_0V$ and $e_1V$ (resp. $f_0V$ and $f_1V$) be the block components of $\Res_{\widetilde{\s}_{n-1}}^{\widetilde{\s}_n}V$ (resp. $\Ind_{\widetilde{\s}_n}^{\widetilde{\s}_{n+1}}V$) with contents $(c_0-1,c_1)$ and $(c_0,c_1-1)$ (resp. $(c_0+1,c_1)$ and $(c_0,c_1+1)$). These blocks components should be thought as $0$ if no blocks with the corresponding content exists. The definitions of $e_0V$, $e_1V$, $f_0V$ and $f_1V$ can then be extended to any module by linearity. Then $\Res_{\widetilde{\s}_{n-1}}^{\widetilde{\s}_n}V\cong e_0V\oplus e_1V$ and $\Ind_{\widetilde{\s}_n}^{\widetilde{\s}_{n+1}}V\cong f_0\oplus f_1V$ by \cite[Theorems 11.2.7, 11.2.8]{KBook}.

By \cite[Theorem 9.2]{JamesBook} and block decomposition we have that:

\begin{lemma}
Let $\la$ be a partition and $i\in\{0,1\}$. Then, in the Grothendieck group,
\[[e_iS^\la]=\sum_A[S^{\la\setminus\{A\}}],\]
where the sum is over all removable nodes $A$ of $\la$ of residue $i$.
\end{lemma}

\begin{lemma}
Let $\la$ be a partition and $i\in\{0,1\}$. Then, in the Grothendieck group,
\[[f_iS^\la]=\sum_A[S^{\la\cup\{A\}}],\]
where the sum is over all addable nodes $A$ of $\la$ of residue $i$.
\end{lemma}

Similarly by \cite[Theorem 2]{mo} or \cite[Theorem 8.1]{s}, Frobenius reciprocity and block decomposition:

\begin{lemma}\label{L240223}
Let $\la\in\Par_2(n)$ and $i\in\{0,1\}$. Then, in the Grothendieck group,
\[[e_iS(\la)]=\sum_A2^{x_A}[S({\la\setminus\{A_i\}})],\]
where the sum is over all bar-removable nodes $A$ of $\la$ of residue $i$ and $x_A=1$ if $A\not=(h(\la),1)$ and $n-h(\la)$ is odd, while $x_A=0$ in all other cases.
\end{lemma}

\begin{lemma}\label{L240223_2}
Let $\la\in\Par_2(n)$ and $i\in\{0,1\}$. Then, in the Grothendieck group,
\[[f_iS(\la)]=\sum_A2^{x_A}[S({\la\cup\{A_i\}})],\]
where the sum is over all bar-addable nodes $A$ of $\la$ of residue $i$ and $x_A=1$ if $A\not=(h(\la)+1,1)$ and $n-h(\la)$ is odd, while $x_A=0$ in all other cases.
\end{lemma}

These 4 lemmas will be used without further reference in the following when computing block components of induced or restricted projective modules.

Partial branching results for projective modules will be given at the end of the next section. In these branching rules normal and conormal nodes appear. As in \cite[Section 11.1]{KBook}, for a given residue $i$ and a partition $\la$, let the $i$-signature consist of a $-$ (resp. $+$) for each removable (resp. addable) $i$-node of $\la$, read from left to right. The reduced $i$-signature is the obtained by recursively removing any $+-$ adjacent pair from the $i$-signature. Nodes corresponding to $-$ (resp. $+$) in the reduced $i$-signature are called normal (resp. conormal). We let $\eps_i(\la)$ (resp. $\phi_i(\la)$) be the number of $i$-normal (resp. $i$-conormal) nodes of $\la$. If $\eps_i(\la)>0$ (resp. $\phi_i(\la)>0$) we further define $\tilde e_i\la$ (resp. $\tilde f_i\la$) to be the partition obtained by removing the leftmost $i$-normal node of $\la$ (resp. adding the rightmost $i$-conormal node of $\la$).

If $\la\in\Par_2(n)$ indexes 2 spin representations, then by \cite[p. 235]{schur} (see also \cite[Theorem 7.1]{s}) the 2-Brauer characters of $S(\la,+)$ and $S(\la,-)$ are equal. Thus $[P^\mu:S(\la,+)]=[P^\mu:S(\la,-)]$ for any $\mu\in\Par_2(n)$. Thus, in the Grothendieck group, any projective module $P$ is a sum (with multiplicities) of some modules $S^\ga$ with $\ga\in\Par(n)$ and some modules $S(\la)$ with $\la\in\Par_2(n)$, so that the multiplicity $[P:S(\la)]$ is well defined. Similarly, for any $G=g_1\ldots g_h$ with $g_i\in\{e_0,e_1,f_0,f_1\}$, the multiplicity $[GS(\nu):S(\la)]$ is well defined in view of Lemmas \ref{L240223} and \ref{L240223_2}.

\section{Projective modules}\label{s3}

Throughout the following let $M_n:=\lfloor n/2\rfloor$ and $m_n:=\lfloor(n-1)/2\rfloor$. Further let $\overline{m}_n:=\lfloor(n-4)/2\rfloor$ if $n\not\equiv 0\Md{4}$ or $\overline{m}_n:=(n-6)/2$ if $n\equiv 0\Md{4}$. Thus $M_n$ is maximal such that $(n-M_n,M_n)$ is a partition and $m_n$ and $\overline{m}_n$ are maximal such that $(n-m_n,m_n)$ and $\dbl(n-\overline{m}_n,\overline{m}_n)$ are 2-regular partitions.

We will now state some basic results which will allow us to compute decomposition numbers.

\begin{lemma}\cite[Theorem 24.15]{JamesBook}\label{e4}
Let $0\leq a\leq M_n$. Any composition factor of $S^{(n-a,a)}$ is of the form $D^{(n-b,b)}$ for some $0\leq b\leq m_n$. Further
\[[S^{(n-a,a)}:D^{(n-b,b)}]=g_{n-2b+1,a-b}.\]
\end{lemma}

\begin{lemma}\label{e1}
Let $P\cong\oplus_\la (P^\la)^{\oplus c_\la}$ be a projective module and $0\leq a\leq M_n$. Then
\begin{align*}
[P:S^{(n-a,a)}]&=\sum_{b=0}^{m_n}c_{(n-b,b)}[P^{(n-b,b)}:S^{(n-a,a)}]\\
&=\sum_{b=0}^{h}c_{(n-b,b)}[P^{(n-b,b)}:S^{(n-a,a)}]
\end{align*}
for any $\min\{a,m_n\}\leq h\leq m_n$.
\end{lemma}

\begin{proof}
This follows from Lemma \ref{e4}.
\end{proof}

\begin{lemma}\label{e2}
Let $P\cong\oplus_\la (P^\la)^{\oplus c_\la}$ be a projective module and $0\leq a\leq m_n$. Then
\begin{align*}
[P:S((n-a,a),\eps)]
=&\sum_{b=0}^{m_n-1}c_{(n-b,b)}[P^{(n-b,b)}:S((n-a,a),\eps)]\\
&+\sum_{c=0}^{\overline{m}_n}c_{\dbl(n-c,c)}[P^{\dbl(n-c,c)}:S((n-a,a),\eps)].
\end{align*}
\end{lemma}

\begin{proof}
In view of Lemmas \ref{t2} and \ref{t3}, composition factors of $S((n-a,a),\eps)$ are of one of the forms $D^{(n-b,b)}$ or $D^{\dbl(n-c,c)}$. The lemma follows since $(n-m_n,m_n)=\dbl(n)$.
\end{proof}

Define $X_n$ to be the set of 2-regular partitions which are not of the forms $(n-c,c)$ or $\dbl(n-c,c)$ for some $c$. In view of Lemmas \ref{e1} and \ref{e2} we define the following subgroups of the Grothendieck group which will be used throughout the paper:

\begin{defi}
For $n\geq 0$ define:
\begin{itemize}
\item $\Tsym_n:=\langle[S^\la]:\la\in\Par(n),h(\la)\geq 3\rangle$,
\item $\Tspin_n:=\langle[S(\la)]:\la\in\Par_2(n),h(\la)\geq 3\rangle$,
\item $T_n:=\langle\Tsym_n,\Tspin_n\rangle$,
\item $R_n:=\langle[P^\la]:\la\in X_n\rangle$.
\end{itemize}
\end{defi}

This set, which is used in the next lemma, will also appear later in the paper.

\begin{lemma}\label{e3}
Let $P$ be a projective module with
\[[P]\equiv\sum_{a=y}^{M_n} \ell_a[S^{(n-a,a)}]+\sum_{b=0}^{m_n}k_b[S((n-b,b))]\Md{T_n}\]
for some $0\leq y\leq m_n-1$. Then
\[[P]\equiv\ell_y[P^{(n-y,y)}]+\sum_{a=y+1}^{m_n-1}\overline{\ell}_a[P^{(n-a,a)}]+\sum_{b=0}^{\overline{m}_n}\overline{k}_b[P^{\dbl(n-b,b)}]\Md{R_n},\]
for some $\overline{\ell}_a,\overline{k}_b$ with $\overline{\ell}_a\leq \ell_a$, $\overline{k}_b\leq k_b/2^{\lfloor h_2((n-b,b))/2\rfloor}$. In particular we have that $\overline{k}_b=0$ whenever $k_b=0$.
\end{lemma}

\begin{proof}
Let $P\cong \oplus_\la(P^\la)^{c_\la}$. We have to check that
\[c_{(n-a,a)}= \left\{\begin{array}{ll}
0,&0\leq a<y,\\
\ell_y,&a=y,\\
\leq \ell_a,&y<a\leq m_n-1
\end{array}\right.\]
and $c_{\dbl(n-b,b)}\leq k_b/2^{\lfloor h_2((n-b,b))/2\rfloor}$ for any $0\leq b\leq \overline{m}_n$.

By Lemma \ref{t2} we have that $[P^{\dbl(n-b,b)}:D^{\dbl(n-b,b)}]=2^{\lfloor h_2((n-b,b))/2\rfloor}$ for any $0\leq b\leq\overline{m}_n$. So the assertion on $c_{\dbl(n-b,b)}$ holds.

For $a<y$ we have that $[P:S^{(n-a,a)}]=0$, so that $c_{(n-a,a)}=0$ in view of Lemma \ref{e4}. For $y\leq a\leq m_n-1$ it then follows from Lemmas \ref{e4} and \ref{e1} that
\begin{align*}
\ell_a&=[P:S^{(n-a,a)}]=\sum_{x=0}^ac_{(n-x,x)}[P^{(n-x,x)}:S^{(n-a,a)}]\\
&=c_{(n-a,a)}+\sum_{x=y}^{a-1}c_{(n-x,x)}[P^{(n-x,x)}:S^{(n-a,a)}],
\end{align*}
so that $c_{(n-a,a)}\leq \ell_a$ with equality holding if $a=y$.
\end{proof}

We will also need the following two lemmas.

\begin{lemma}\label{L100822}
If $r\geq 1$, $\mu\in\Par_p(n-r)$, $\la\in\Par_p(n)$ and $i,i_1,\ldots,i_r\in\Z/p\Z$. Then
\[[f_{i_1}\ldots f_{i_r}P^\mu:P^\la]=[e_{i_r}\ldots e_{i_1}D^\la:D^\mu].\]
In particular if $\eps_i(\la)\leq\eps_i(\mu)+r$ and $[f_i^rP^\mu:P^\la]>0$ then $\la=\tilde f_i^r\mu$, in which case
\[[f_i^rP^\mu:P^{\tilde f_i^r\mu}]=r!\binom{\eps_i(\mu)+r}{r}.\]
\end{lemma}

\begin{proof}
Since
\begin{align*}
[f_{i_1}\ldots f_{i_r}P^\mu:P^\la]&=\sum_\nu[f_{i_1}P^\mu:P^\nu][f_{i_2}\ldots f_{i_r}P^\nu:P^\la],\\
[e_{i_r}\ldots e_{i_1}D^\la:D^\mu]&=\sum_\nu[e_{i_r}\ldots e_{i_2}D^\la:D^\nu][e_{i_1}D^\nu:D^\mu],
\end{align*}
in the first statement we may assume $r=1$, in which case it holds by Frobenius reciprocity and block decomposition. The second statement then follows from \cite[Theorem 11.2.10]{KBook}.
\end{proof}

\begin{lemma}\label{L100822_2}
If $r\geq 1$, $\mu\in\Par_2(n+r)$, $\la\in\Par_2(n)$ and $i,i_1,\ldots,i_r\in\Z/p\Z$. Then
\[[e_{i_1}\ldots e_{i_r}P^\mu:P^\la]=[f_{i_r}\ldots f_{i_1}D^\la:D^\mu].\]
In particular if $\phi_i(\la)\leq\phi_i(\mu)+r$ and $[e_i^rP^\mu:P^\la]>0$ then $\la=\tilde e_i^r\mu$, in which case
\[[e_i^rP^\mu:P^{\tilde e_i^r\mu}]=r!\binom{\phi_i(\mu)+r}{r}.\]
\end{lemma}

\begin{proof}
Similar to the previous lemma, using \cite[Theorem 11.2.11]{KBook} instead.
\end{proof}

\section{Proof of Theorem \ref{t1}}\label{s4}

In this section we will now prove Theorem \ref{t1}. The cases $n\leq 13$ can be checked using known decomposition matrices \cite{gap,ModularAtlas} (using block decomposition, dimension and type of modules to identify characteristic 0 modules and using Lemma \ref{e4} to identify the corresponding columns of the decomposition matrix). We will first prove Theorem \ref{t1} for $n$ odd by induction and then use this to prove it when $n$ is even.

{\bf Case 1:} $n$ is odd.  We may assume that $n\geq 15$ is odd and that Theorem \ref{t1} holds for $n-2$.

{\bf Case 1.1:} $b>0$. Let $i$ be the residue of the addable nodes in the first 2 rows of $(n-b-1,b-1)$ (these 2 nodes have the same residue since $n$ is odd). If $S^{(n-a-2,a)}$ is in the same block as $D^{(n-b-1,b-1)}$ then the addable nodes in the first two rows of $(n-a-2,a)$ both have residue $i$. Further by induction and Lemma \ref{e4} we have that
\begin{align*}
[P^{(n-b-1,b-1)}]\equiv&[S^{(n-b-1,b-1)}]+\sum_{a=b}^{(n-3)/2}g_{n-2a-1,a-b+1}[S^{(n-a-2,a)}]\\
&+2^{d_{n-2b+1}}[S((n-c-2,c))]\Md{T_n}
\end{align*}
with $c$ equal to $(n-3)/2$ or $(n-5)/2$ (depending on $n$ and $b$). So
\begin{align}
[f_i^2P^{(n-b-1,b-1)}]\equiv&2[S^{(n-b,b)}]+\sum_{a=b}^{(n-3)/2}2g_{n-2a-1,a-b+1}[S^{(n-a-1,a+1)}]\label{e6}\\
&+2^{d_{n-2b+1}}[f_i^2S((n-c-2,c))]\Md{T_n}.\nonumber
\end{align}
Since
\[[S^{(n-a-1,a+1)}:D^{(n-b,b)}]=[S^{(n-a-2,a)}:D^{(n-b-1,b-1)}]\]
we then have by Lemma \ref{e3} that
\begin{equation}\label{e7}
[f_i^2P^{(n-b-1,b-1)}]\equiv2[P^{(n-b,b)}]+\sum_{a=0}^{\overline{m}_n}\overline{k}_a[P^{\dbl(n-a,a)}]\Md{R_n}
\end{equation}
where, for any $0\leq a\leq\overline{m}_n$, if $x_a$ the multiplicity of $S((n-a,a),\eps)$ then $f_i^2S((n-c-2,c))$ $\overline{k}_a\leq 2^{d_{n-2b+1}}x_a$. In particular if
\begin{align*}
[f_i^2S((n-c-2,c))]\equiv&x_{(n-1)/2}[S(((n+1)/2,(n-1)/2))]\\
&+x_{(n-3)/3}[S((n+3)/2,(n-3)/2))]\Md{\Tspin_n}
\end{align*}
then
\[[f_i^2P^{(n-b-1,b-1)}]\equiv2[P^{(n-b,b)}]\Md{R_n}\]
and so by Lemma \ref{e2}, \eqref{e6} and \eqref{e7}
\begin{align*}
[S((n-a,a),\eps):D^{(n-b,b)}]&=[f_i^2P^{(n-b-1,b-1)}:S((n-a,a),\eps)]\\
&=[f_i^2P^{(n-b-1,b-1)}:S((n-a,a),\eps)]/2\\
&=2^{d_{n-2b+1}}[f_i^2S((n-c-2,c)):S((n-a,a),\eps)]/2.
\end{align*}
We will now consider different cases, starting with those where this argument allows to prove the theorem.

{\bf Case 1.1.1:} $n\equiv 1\Md{8}$ and $b$ is even. Then $i=0$ and $c=(n-3)/2$. So
\[[f_i^2S((n-c-2,c))]=2[S(((n+1)/2,(n-1)/2))]+4[S(((n+1)/2,(n-3)/2,1))].\]

{\bf Case 1.1.2:} $n\equiv 3\Md{8}$ and $b$ is even. Then $i=0$ and $c=(n-5)/2$. So
\[[f_i^2S((n-c-2,c))]=2[S(((n+1)/2,(n-1)/2))]+4[S(((n+1)/2,(n-3)/2,1))].\]

{\bf Case 1.1.3:} $n\equiv 3\Md{8}$ and $b$ is odd. Then $i=1$ and $c=(n-3)/2$. So
\[[f_i^2S((n-c-2,c))]=2[S(((n+3)/2,(n-3)/2))].\]

{\bf Case 1.1.4:} $n\equiv 5\Md{8}$ and $b$ is odd. Then $i=1$ and $c=(n-3)/2$. So
\[[f_i^2S((n-c-2,c))]=2[S(((n+1)/2,(n-1)/2))].\]

{\bf Case 1.1.5:} $n\equiv 7\Md{8}$ and $b$ is even. Then $i=0$ and $c=(n-3)/2$. So
\[[f_i^2S((n-c-2,c))]=2[S(((n+3)/2,(n-3)/2))]+4[S(((n+1)/2,(n-3)/2,1))].\]

{\bf Case 1.1.6:} $n\equiv 7\Md{8}$ and $b$ is odd. Then $i=1$ and $c=(n-5)/2$. So
\[[f_i^2S((n-c-2,c))]=2[S(((n+1)/2,(n-1)/2))].\]

{\bf Case 1.1.7:} $n\equiv 1\Md{8}$ and $b$ is odd. Then $i=1$ and $c=(n-1)/2$. So
\[[f_i^2S((n-c-2,c))]=4[S(((n+3)/2,(n-3)/2))]+2[S(((n+5)/2,(n-5)/2))]\]
and then
\[[f_i^2P^{(n-b-1,b-1)}:S((n-a,a),\eps)]=\left\{\begin{array}{ll}
2^{d_{n-2b+1}+2}&a=(n-3)/2,\\
2^{d_{n-2b+1}+1}&a=(n-5)/2,\\
0,&\text{else.}
\end{array}
\right.
\]
So we only need to check the multiplicity of $D^{(n-b,b)}$ as a composition factor of $S(((n+3)/2,(n-3)/2),\pm)$ and $S(((n+5)/2,(n-5)/2),\pm)$.

{\bf Case 1.1.7.1:} $b\leq(n-7)/2$. Considering $P^{(n-b-2,b)}$ we have that
\begin{align*}
[P^{(n-b-2,b)}]\equiv&[S^{(n-b-2,b)}]+\sum_{a=b+1}^{(n-3)/2}g_{n-2b-1,a-b}[S^{(n-a-2,a)}]\\
&+2^{d_{n-2b-1}}[S(((n-1)/2,(n-3)/2))]\Md{T_{n-2}}.
\end{align*}
So
\begin{align*}
[f_1f_0P^{(n-b-2,b)}]\equiv&[S^{(n-b,b)}]+\sum_{a=b+1}^{(n-1)/2}H'_a[S^{(n-a,a)}]\\
&+2^{d_{n-2b-1}+1}[S(((n+3)/2,(n-3)/2))]\Md{T_n}.
\end{align*}
So $P^{(n-b,b)}$ is a direct summand of $f_1f_0P^{(n-b-2,b)}$ by Lemma \ref{e3}. Since $f_1f_0P^{(n-b-2,b)}$ has no composition factor $S(((n+5)/2,(n-5)/2),\pm)$ the same holds also for $P^{(n-b,b)}$ (and then $D^{(n-b,b)}$ is not a composition factor of $S(((n+5)/2,(n-5)/2),\pm)$). 

By Lemma \ref{t2}
\begin{align*}
[S(((n+5)/2,(n-5)/2),\pm):D^{\dbl((n+5)/2,(n-5)/2)}]&=1,\\
[S(((n+3)/2,(n-3)/2),\pm):D^{\dbl((n+5)/2,(n-5)/2)}]&=1.
\end{align*}
So by Lemma \ref{e3}
\[[f_i^2P^{(n-b-1,b-1)}]\equiv2[P^{(n-b,b)}]+2^{d_{n-2b+1}+1}[P^{\dbl((n+5)/2,(n-5)/2)}]\Md{R_n}.\]
and then by Lemma \ref{e1}
\begin{align*}
&[S(((n+3)/2,(n-3)/2),\pm):D^{(n-b,b)}]\\
&=1/2[f_i^2P^{(n-b-1,b-1)}:S(((n+3)/2,(n-3)/2),\pm)]\\
&\hspace{11pt}-2^{d_{n-2b+1}}[P^{\dbl((n+5)/2,(n-5)/2)}:S(((n+3)/2,(n-3)/2),\pm)]\\
&=2^{d_{n-2b+1}+1}-2^{d_{n-2b+1}}\\
&=2^{d_{n-2b+1}}.
\end{align*}

{\bf Case 1.1.7.2:} $b=(n-3)/2$. 
By Lemma \ref{e4} and induction
\begin{align*}
[P^{((n+3)/2,(n-7)/2)}]\equiv\,&[S^{((n+3)/2,(n-7)/2)}]+[S^{((n-1)/2,(n-3)/2)}]\\
&+2[S(((n-1)/2,(n-3)/2))]\Md{T_{n-2}}
\end{align*}
and then
\begin{align*}
[f_1f_0P^{((n+3)/2,(n-7)/2)}]\equiv\,&[S^{((n+7)/2,(n-7)/2)}]+2[S^{((n+3)/2,(n-3)/2)}]\\
&+4[S(((n+3)/2,(n-3)/2))]\Md{T_n}.
\end{align*}
By Lemmas \ref{e4} and \ref{e1} we then have that $P^{((n+3)/2,(n-3)/2)}$ is a direct summand of $f_1f_0P^{((n+3)/2,(n-7)/2)}$. Since
\[[f_1f_0P^{((n+3)/2,(n-7)/2)}:S(((n+5)/2,(n-5)/2),\pm)]=0\]
we can then conclude as in the previous case.

{\bf Case 1.1.8:} $n\equiv 5\Md{8}$ and $b$ is even. Then $i=0$ and $c=(n-5)/2$. So
\begin{align*}
&[f_i^2S((n-c-2,c))]\\
&=4[S(((n+3)/2,(n-3)/2))]+2[S(((n+5)/2,(n-5)/2))]\\
&\hspace{11pt}+4[S(((n+1)/2,(n-3)/2,1))]+4[S(((n+3)/2,(n-5)/2,1))].
\end{align*}
and then
\[[f_i^2P^{n-b-1,b-1)}:S((n-a,a),\eps)]=\left\{\begin{array}{ll}
2^{d_{n-2b+1}+2}&a=(n-3)/2,\\
2^{d_{n-2b+1}+1}&a=(n-5)/2,\\
0,&\text{else.}
\end{array}
\right.
\]
Again we only need to check the multiplicity of $D^{(n-b,b)}$ as a composition factor of $S(((n+3)/2,(n-3)/2),\pm)$ and $S(((n+5)/2,(n-5)/2),\pm)$.

Note that $b\leq (n-5)/2$ since $b\leq (n-3)/2$ is even. This case can be checked similarly to Case 1.1.7.1 using $f_0f_1P^{(n-b-2,b)}$.

{\bf Case 1.2:} $b=0$. In this case we will study $f_0f_1P^{(n-2)}$. Let $c=(n-3)/2$ if $n\equiv \pm 3\Md{8}$ or $c=(n-5)/2$ if $n\equiv\pm 1\Md{8}$. Then
\begin{align*}
[P^{(n-2)}]\equiv&\sum_{a=0}^{(n-3)/2}g_{n-1,a}[S^{(n-a-2,a)}]+2^{d_{n-1}}[S((n-c-2,c))]\Md{T_{n-2}}.
\end{align*}
Note that (by definition of $g_{n-1,a}$ or by block decomposition) $g_{n-1,a}=0$ if $a$ is odd. For $a$ even we have that
\[[f_0f_1S^{(n-a-2,a)}]\equiv[S^{(n-a,a)}]+\de_{a<(n-3)/2}[S^{(n-a-2,a+2)}]\Md{\Tsym_n}.\]

For $c'=(n-1)/2$ if $n\equiv 1\text{ or }3\Md{8}$ or $c'=(n-3)/2$ if $n\equiv 5\text{ or }7\Md{8}$ it can be checked that
\[[f_0f_1S((n-c-2,c))]\equiv2[S((n-c',c'))]\Md{\Tspin_n}.\]
So, by Lemma \ref{e3},
\begin{align*}
[f_0f_1P^{(n-2)}]\equiv\sum_{a=0}^{(n-1)/2}L_a[P^{(n-a,a)}]+\sum_{\la\in X_n}L_\la[P_\la]
\end{align*}
for some $L_a$. Let $k\geq 1$ maximal with $n+1=h2^k$ for some $h>1$ (so  that $h=2$ or $h\geq 3$ is odd).

We will first show that
\begin{align}\label{E270223}
[f_0f_1P^{(n-2)}]\equiv[P^{(n)}]+2\sum_{i=1}^{k-1}[P^{(n-2^i,2^i)}]\Md{R_n}
\end{align}
and then use this to show that $[P^{(n)}:S((n-c',c'))]=2^{d_{n+1}}$.

To prove \eqref{E270223} it is enough to check the multiplicities of $P^{(n-a,a)}$ in $f_0f_1P^{(n-2)}$ for $0\leq a\leq m_n$. In view of Lemma \ref{L100822} this is equivalent to showing that
\[[e_1e_0D^{(n-a,a)}:D^{(n-2)}]=\left\{\begin{array}{ll}
1,&a=0,\\
2,&a=2^i\text{ with }1\leq i\leq k-1,\\
0,&\text{else.}
\end{array}\right.\]
Since any composition factor of $e_0D^{(n-a,a)}$ is of the form $D^{(n-b-1,b)}$ for some $b$ and by \cite[Theorem 11.2.7]{KBook} $[e_1D^{(n-b-1,b)}:D^{(n-2)}]=\de_{b,0}$ (using that $n$ is odd, so that $(n-b-1,b)$ has only one normal node), it follows that $[e_1e_0D^{(n-a,a)}:D^{(n-2)}]=[e_0D^{(n-a,a)}:D^{(n-1)}]$. We will use the main theorem (on p.3304) of \cite{sh} without further reference until the end of Case 1.2. By block decomposition we have that
\[[e_0D^{(n-a,a)}:D^{(n-1)}]=\left\{\begin{array}{ll}
1,&a=0,\\
0,&a>0\text{ with }a\not=2^i\text{ for some }i\geq 1.
\end{array}\right.\]
So we may assume that $a=2^i$ with $i\geq 1$. Note that
\[n-2a=h2^k-2^{i+1}-1=(h-1)2^k-2^{i+1}+\sum_{j=0}^{k-1}2^j.\]

If $i\geq k$ then $h\geq 3$ is odd (since $2^i<n/2$ as $(n-2^i,2^i)$ is 2-regular). It follows that $2^k$ is the smallest power of 2 missing in the 2-adic decomposition of $n-2a$ and so $[e_0D^{(n-2^i,2^i)}:D^{(n-1)}]=0$.

If $i=k-1$ then $n-2a=(h-2)2^k+\sum_{j=0}^{k-1}2^j=(h-2)2^{i+1}+\sum_{j=0}^{i}2^j$. Since $\sum_{j=0}^{i}2^j$ appears in the 2-adic decomposition of $n-2a$, it follows that $[e_0D^{(n-2^i,2^i)}:D^{(n-1)}]=2$.

If $i\leq k-2$ then $n-2a=(h-1)2^k+\sum_{j=i+2}^{k-1}2^j+\sum_{j=0}^{i}2^j$, so again $[e_0D^{(n-2^i,2^i)}:D^{(n-1)}]=2$.

We will now use \eqref{E270223} to show that $[P^{(n)}:S((n-c',c'))]=2^{d_{n+1}}$. To see this, note that $d_{n+1}=d_h+k$. Further for $0\leq i\leq k-1$ we have
\[n-2^{i+1}+1=(h-1)2^k+\sum_{j=i+1}^{k-1}2^j.\]
Since $h\geq 2$ we then have that
\[d_{n-2^{i+1}+1}=
d_{h-1}+2k-i-1=d_h+2k-i-2,\]
where the last equality holds since $h=2$ or $h\geq 3$ is odd. So
\begin{align*}
2^{d_{n-1}+1}-\sum_{i=1}^{k-1}2^{d_{n-2^{i+1}+1}+1}&=2^{d_h+2k-1}-\sum_{i=1}^{k-1}2^{d_h+2k-i-1}\\
&=2^{d_h+2k-1}-2^{d_h+2k-1}+2^{d_h+k}=2^{d_h+k}=2^{d_{n+1}}.
\end{align*}

{\bf Case 2:} $n$ is even. In this case $n-2b\geq 4$, so $n-2b-1\geq 3$. By case 1 we have that
\begin{align*}
[P^{(n-b-1,b)}]\equiv&\sum_{a=b}^{M_{n-1}}g_{n-2b,a-b}[S^{(n-a-1,a)}]\\
&+2^{d_{n-2b}}[S((n-c-1,c))]\Md{T_{n-1}}
\end{align*}
with $c=n/2-1$ or $n/2-2$ depending on $n$ and $b$. Let $i$ be the residue of the two top addable nodes of $(n-b-1,b)$. If $g_{n-2b,a-b}=1$ then $S^{(n-a-1,a)}$ is in the same block of $D^{(n-b-1,b)}$ and so the two top addable nodes of $(n-a+1,a)$ have residue $i$. Then
\begin{align*}
[f_iP^{(n-b-1,b)}]\equiv&\sum_{a=b+1}^{M_{n-1}}g_{n-2b,a-b}([S^{(n-a,a)}]+[S^{(n-a-1,a+1)}])\\
&+2^{d_{n-2b}}[f_iS((n-c-1,c))]\Md{T_n}.
\end{align*}
Set $S^{(x,y)}:=0$ if $x<y$. For $a>n/2=M_{n-1}+1$ we have that
\[[S^{(n-a,a)}]+[S^{(n-a-1,a+1)}]=0,\]
while for $a=n/2$ we have that $g_{n-2b,a-b}=0$ since $n-2b=2(a-b)>0$. So 
\begin{align*}
[f_iP^{(n-b-1,b)}]\equiv&\sum_{a\geq b}g_{n-2b,a-b}([S^{(n-a,a)}]+[S^{(n-a-1,a+1)}])\\
&+2^{d_{n-2b}}[f_iS((n-c-1,c))]\Md{T_n}.
\end{align*}
Note that if $g_{n-2b,a-b}=1$ then $a-b$ is even. For $a\leq n/2$ we have that $a<n-b$ as $n-2b\geq 4$, so that $a-b<n-2b$. Thus if $a\leq n/2$ with $a-b$ even then $g_{n-2b,a-b}=g_{n-2b+1,a-b}=g_{n-2b+1,a+1-b}$. Again using that $[S^{(n-a,a)}]+[S^{(n-a-1,a+1)}]=0$ for $a>n/2$ it follows that
\begin{align*}
&[f_iP^{(n-b-1,b)}]\equiv\sum_{a\geq b}g_{n-2b+1,a-b}[S^{(n-a,a)}]+2^{d_{n-2b}}[f_iS((n-c-1,c))]\\
&\equiv\sum_{a=b}^{M_n}g_{n-2b+1,a-b}[S^{(n-a,a)}]+2^{d_{n-2b}}[f_iS((n-c-1,c))]\Md{T_n}.
\end{align*}
Considering the values of $i$ and $c$ in each case, it can be checked that if $n\equiv 0\Md{8}$ and $b$ is even or if $n\equiv 4\Md{8}$ and $b$ is odd then
\begin{align*}
[f_iS((n-c-1,c))]\equiv\,&2[S((n/2+1,n/2-1))]\\
&+2[S((n/2+2,n/2-2))]\Md{\Tspin_n}
\end{align*}
while in all other cases
\[[f_iS((n-c-1,c))]\equiv2[S((n/2+1,n/2-1))]\Md{\Tspin_n}.\]
In view of Lemmas \ref{e4} and \ref{e3} it follows that
\[[f_iP^{(n-b-1,b)}]\equiv[P^{(n-b,b)}]\Md{R_n}.\]
Since $d_{n-2b+1}=d_{n-2b}+1$ (by definition), the theorem then follows from Lemma \ref{e2}.

\section{Proof of Theorem \ref{T100822}}\label{s5}

In this section we will prove Theorem \ref{T100822}. We start with the case $n\equiv 2\Md{4}$ and $b$ odd (this is equivalent to $\dbl(n-b,b)$ having 2-core $(4,3,2,1)$ or $(3,2,1)$).

\begin{lemma}\label{T100822_2}
Let $p=2$, $n\equiv 2\Md{4}$, $0\leq a\leq m_n$ and $1\leq b\leq \overline{m}_n$ with $b$ odd. Then
\[[S((n-a,a),0):D^{\dbl(n-b,b)}]=g_{n-2b,a-b}.\]
\end{lemma}

\begin{proof}
If $a$ is even then $S((n-a,a),0)$ and $D^{\dbl(n-b,b)}$ are in different blocks. Since $n-2b$ is even while $a-b$ is odd, the theorem holds in this case. So we may assume that $a$ is odd.

By Lemma \ref{e4} we have that for any $0\leq x\leq M_n$ and $0\leq y\leq m_n$
\[[S^{(n-x,x)}:D^{(n-y,y)}]=g_{n-2y+1,x-y}.\]
By \cite[Theorem 9.3]{s}, in the Grothendieck group,
\[[S^{(n-x,x)}\otimes S((n))]=(1+\de_{x\not=0})[S((n-x,x))]+(1+\de_{x\not=1})[S((n-x+1,x-1))],\]
where $S((n/2,n/2))$ and $S((n+1,-1))$ are both defined to be 0. It follows that for $1\leq a\leq n/2-2$ odd
\[[S((n-a,a),0)]=1/2\sum_{c=0}^{(a-1)/2}([S^{(n-2c-1,2c+1)}\otimes S((n))]-[S^{(n-2c,2c)}\otimes S((n))].\]
If $y>2c+1$ then $D^{(n-y,y)}$ is not a composition factor of $S^{(n-2c-1,2c+1)}$ or $S^{(n-2c,2c)}$. If $y\leq 2c+1$ is even then
\begin{align*}
[S^{(n-2c-1,2c+1)}:D^{(n-y,y)}]&=g_{n-2y+1,2c-y+1}=g_{n-2y+1,2c-y}\\
&=[S^{(n-2c,2c)}:D^{(n-y,y)}]
\end{align*}
as both $n-2y$ and $2c-y$ are even. Thus
\begin{align*}
&[S((n-a,a),0)]\\
&=1/2\sum_{c=0}^{(a-1)/2}\sum_{z=0}^c(g_{n-4z-1,2c-2z}-g_{n-4z-1,2c-2z-1})[D^{(n-2z-1,2z+1)}\otimes S((n))]\\
&=\sum_{c=0}^{(a-1)/2}\sum_{z=0}^c(g_{n-4z-1,2c-2z}-g_{n-4z-1,2c-2z-1})[D^{\dbl(n-2z-1,2z+1)}],\\
&=\sum_{z=0}^{(a-1)/2}\sum_{c=z}^{(a-1)/2}(g_{n-4z-1,2c-2z}-g_{n-4z-1,2c-2z-1})[D^{\dbl(n-2z-1,2z+1)}]
\end{align*}
with the second equality holding by \cite[Corollary 3.21]{gj}. All partitions appearing are 2-regular, since $2c+1\leq a\leq \overline{m}_n<m_n$ and $0\leq z\leq c$.

Since $n-4z$ and $2c-2z$ are even
\begin{align*}
&\sum_{c=z}^{(a-1)/2}(g_{n-4z-1,2c-2z}-g_{n-4z-1,2c-2z-1})\\
&=\sum_{c=z}^{(a-1)/2}(g_{n-4z-2,2c-2z}-g_{n-4z-2,2c-2z-2})\\
&=g_{n-4z-2,a-2z-1}-g_{n-4z-2,-2}\\
&=g_{n-4z-2,a-2z-1}
\end{align*}
and so the theorem holds.
\end{proof}

In order to prove Theorem \ref{T100822} in general, we will use some block components of inductions/restrictions of the modules $P^{\dbl(a,b)}$ with $a\equiv b\equiv\pm 1\Md{4}$. We point out that the cases $b$ odd if $n\equiv 0\Md{4}$ or $b$ even if $n\equiv 2\Md{4}$ are not covered in the following lemma.

\begin{lemma}\label{L150223}
Let $1\leq b\leq \overline{m}_n$. Let $i=0$ if $b\equiv 0\text{ or }3\Md{4}$ or $i=1$ if $b\equiv 1\text{ or }2\Md{4}$. Define $F$, $C$ and $x$ and $y$ through:
\begin{itemize}
\item if $n\equiv 0\Md{4}$ and $b$ is odd then $x=2$, $y=0$, $F=f_i^2$ and $C=2$,

\item if $n\equiv 1\Md{4}$ and $b$ is even then $x=2$, $y=1$, $F=f_i^3$ and $C=6$,

\item if $n\equiv 1\Md{4}$ and $b$ is odd then $x=3$, $y=0$, $F=e_if_{1-i}f_i^3$ and $C=6(2+\de_{i=0})$,

\item if $n\equiv 2\Md{4}$ and $b$ is even then $x=3$, $y=1$, $F=f_{1-i}f_i^3$ and $C=6$,

\item if $n\equiv 3\Md{4}$ and $b$ is even then $x=4$, $y=1$, $F=f_{1-i}^2f_i^3$ and $C=12$,

\item if $n\equiv 3\Md{4}$ and $b$ is odd then $x=1$, $y=0$, $F=f_i$ and $C=1$.
\end{itemize}
Then 
\[[FP^{\dbl(n-b-x,b-y)}]\equiv C[P^{\dbl(n-b,b)}]\Md{R_n}.\]
\end{lemma}

\begin{proof}
This can be seen by (repeated) application of the following argument and by comparing numbers of (co)normal nodes of all 2-regular partitions of the forms $(A,B)$, $\dbl(A,B)$ and in the last case also $(A,B,D)$ or $\dbl(A,B,1)$ with the right content.

Let $j$ be a residue, $r\geq 1$ and $\al=\dbl(N-B,B)$. 
If
\[[Q]\equiv d[P^\al]\Md{R_N}\]
and that $\eps_j(\al)\geq\eps_j(\be)-r$ for all 2-regular partitions $\be$ of $N+r$ which are 2-parts partitions or doubles of 2-parts partitions with $P^\be$ in the same block as $f_j^r P^\al$. By Lemmas \ref{t3} and \ref{e4} we have that no composition factor of $\sum_{\mu\in X_N}d_\ga[P^\ga]$ is of the forms $S^{(N-e,e)}$ or $S((N-e,e),\ldots)$, so no composition factor of $\sum_{\mu\in X_N}d_\ga[f_j^rP^\ga]$ is of the forms $S^{(N+r-e,e)}$ or $S((N+r-e,e),\ldots)$ and then by Lemma \ref{L100822}
\[[f_j^rQ]\equiv dr!\binom{\eps_j(\al)+r}{r}[P^{\tilde f_j^r\al}]\Md{R_{N+r}}.\]

Similarly by Lemma \ref{L100822_2} if $\phi_1(\al)\geq\phi_1(\be)-r$ for all 2-regular partitions $\be$ of $N-r$ which are 2-parts partitions or doubles of 2-parts partitions with $P^\be$ in the same block as $e_1^r P^\al$, then
\[[e_1^rQ]\equiv dr!\binom{\phi_1(\al)+r}{r}[P^{\tilde e_1^r\al}]\Md{R_{N-r}}\]
(as we are removing nodes of residue 1).

When taking $e_0f_1f_0^3 P^{\dbl(e,f)}$ a more careful analysis is needed. In this case $n\equiv 1\Md{4}$ and $b\equiv 3\Md{4}$, so that $e\equiv f\equiv 3\Md{4}$. It can be checked that
\begin{align*}
[f_0^3P^{\dbl(e,f)}]&=6[P^{\dbl(e+2,f+1)}]+[P]\\
[f_1f_0^3P^{\dbl(e,f)}]&=6[P^{\dbl(e+3,f+1)}]+[Q]
\end{align*}
with no projective module indexed by a partition with at most 3 rows or of the form $\dbl(E,F)$ or $\dbl(E,F,1)$ appearing in either $P$ or $Q$. So no module of the form $S^{(g,h)}$ or $S((g,h),\ldots)$ appears in $e_0Q$, which allows to show that
\[[e_0f_1f_0^3P^{\dbl(e,f)}]\equiv 12[P^{\dbl(e+3,f)}]\Md{R_{e+f+3}}.\]
\end{proof}

We are now ready to prove Theorem \ref{T100822}.

\begin{proof}[Proof of Theorem \ref{T100822}]
If $n\equiv 2\Md{4}$ and $b$ is odd the theorem has been proved in Lemma \ref{T100822_2}. Let $i=0$ if $b\equiv 0\text{ or }3\Md{4}$ while $i=1$ if $b\equiv 1\text{ or }2\Md{4}$. By Lemma \ref{L150223} we have that
\[[FP^{\dbl(n-b-x,b-y)}]\equiv C[P^{\dbl(n-b,b)}]\Md{R_n}\]
for $F$, $C$ and $x$ and $y$ given by:
\begin{itemize}
\item if $n\equiv 0\Md{4}$ and $b$ is odd then $x=2$, $y=0$, $F=f_i^2$ and $C=2$,

\item if $n\equiv 1\Md{4}$ and $b$ is even then $x=2$, $y=1$, $F=f_i^3$ and $C=6$,

\item if $n\equiv 1\Md{4}$ and $b$ is odd then $x=3$, $y=0$, $F=e_if_{1-i}f_i^3$ and $C=6(2+\de_{i=0})$,

\item if $n\equiv 2\Md{4}$ and $b$ is even then $x=3$, $y=1$, $F=f_{1-i}f_i^3$ and $C=6$,

\item if $n\equiv 3\Md{4}$ and $b$ is even then $x=4$, $y=1$, $F=f_{1-i}^2f_i^3$ and $C=12$,

\item if $n\equiv 3\Md{4}$ and $b$ is odd then $x=1$, $y=0$, $F=f_i$ and $C=1$.
\end{itemize}
In view of Lemmas \ref{L100822} and \ref{L100822_2} we then have that in each case
\begin{align*}
&[S((n-a,a),\eps):D^{\dbl(n-b,b)}]\\
&=[P^{\dbl(n-b,b)}:S((n-a,a),\eps)]\\
&=1/C[FP^{\dbl(n-b-x,b-y)}:S((n-a,a),\eps)]\\
&=1/C\sum_\nu[P^{\dbl(n-b-x,b-y)}:S(\nu)][FS(\nu):S((n-a,a),\eps)]\\
&=1/C\sum_\nu[P^{\dbl(n-b-x,b-y)}:S(\nu)][FS(\nu):S((n-a,a))].
\end{align*}

Apart from the case $n\equiv 1\Md{4}$ and $b$ odd, $F=f_{i_1}\ldots f_{i_r}$ for some $r$ and $i_1,\ldots,i_r$. So we may limit the sum to partitions $\nu$ with at most 2-parts.

If $n\equiv 1\Md{4}$ and $b\equiv 1\Md{4}$ then $[P^{\dbl(n-b,b)}:S((n-a,a),\eps)]$ is obtained considering $e_1f_0f_1^3P^{\dbl(n-b-3,b)}$. Since removing 1-nodes from bar-partitions does not change their length, we may again restrict the sum to $\nu$ with at most 2-parts.

If $n\equiv 1\Md{4}$ and $b\equiv 3\Md{4}$ and $S((r,s,t),\eps'')$ is any module appearing in $P^{\dbl(n-b-3,b)}$ with $t\geq 1$ then $t>1$ by block decomposition (comparing bar-contents it can be checked that 2 of $r,s,t$ are $\equiv 3\Md{4}$ and the third is $\equiv 0\Md{4}$), so these modules do not give rise to modules of the form $S((n-a,a),\eps)$ in $e_0f_1f_0^3P^{\dbl(n-b-3,b)}$. So also in this case we only need to consider partitions $\nu$ with at most 2 parts.

Note that in each case $n-b-x-y\equiv 2\Md{4}$ and $b-y$ is odd. Thus
\begin{align*}
&[P^{\dbl(n-b-x,b-y)}:S((n-x-y-z,z))]\\
&=[P^{\dbl(n-b-x,b-y)}:S((n-x-y-z,z),\eps')]\\
&=g_{n-2b-x+y,z-b+y}
\end{align*}
by Lemma \ref{T100822_2}. Since $n-2b-x+y$ is even, if $g_{n-2b-x+y,z-b+y}\not=0$ then $z\equiv b-y\Md{2}$ is odd. So $[S((n-a,a),\eps):D^{\dbl(n-b,b)}]$ is given by
\begin{align*}
\sum_{z=0}^{(n-x-y-6)/4}\frac{g_{n-2b-x+y,2z-b+y+1}}{C}[FS((n-x-y-2z-1,2z+1)):S((n-a,a))]
\end{align*}
(since $\overline{m}_{n-x-y}=(n-x-y-6)/4$). If we set $S(\la):=0$ whenever $\la$ is not a 2-regular partition then $[S((n-a,a),\eps):D^{\dbl(n-b,b)}]$ is given by
\begin{align*}
\sum_{z\in\Z} \frac{g_{n-2b-x+y,2z-b+y+1}}{C}[FS((n-x-y-2z-1,2z+1)):S((n-a,a))].
\end{align*}

If $z<0$ then $2z-b+y-1<0$, while if $z=(n-x-y-2)/4$ then $2z-b+y+1=(n-2b-x+y)/2$. In either of these two cases $g_{n-2b-x+y,2z-b+y+1}=0$.

If $z\geq(n-x-y+2)/4$ then $n-x-y-2z-1\geq 2z+5$ and so $(n-x-y-2z+k-1,2z+\ell+1)$ is not a 2-regular partition for any $k\leq 4$ and $\ell\geq 0$.

Further, since $n-2b-x+y\equiv 0\Md{4}$, if $g_{n-2b-x+y,2z-b+y+1}=1$ then the partition $(n-x-y-2z-1,2z+1)$ has 2 (recursively) bar-addable nodes of residue $i$ on each of the first and second row (this could also be seen using block decomposition and comparing bar-contents) and that $n-x-y-2z-1\geq2z+5$ since $z\leq (n-x-y-6)/4$.

These facts can be used to check that
\[g_{n-2b-x+y,2z-b+y+1}[FS((n-x-y-2z-1,2z+1))]=0\]
whenever $S(\la)=0$.

We will now in each of the 6 cases compute
\[\sum_{z\in\Z} \frac{g_{n-2b-x+y,2z-b+y+1}}{C}[FS((n-x-y-2z-1,2z+1))]\]
and study the coefficient of $S((n-a,a))$ to prove the theorem.

Since $b\leq\overline{m}_n$, we always have that $n-2b-1\geq 4$ if $n\equiv 1\Md{4}$ and $b$ is even or $n\equiv 3\Md{4}$ and $b$ is odd, $n-2b-2\geq 4$ if $n\equiv 0\Md{4}$ and $b$ is odd or $n\equiv 2\Md{4}$ and $b$ is even and $n-2b-3\geq 4$ if $n\equiv 1\Md{4}$ and $b$ is odd or $n\equiv 3\Md{4}$ and $b$ is even. This will be used without further reference in the following case analysis to compare the existence of some $h$ with $r\leq 2^h\leq s$ whenever comparing distinct coefficients $g_{r,s}$. Comparing the 2-adic decompositions (or at least the last two summands in each of them) of the corresponding $r$ and $s$ can be done by writing each appearing number as $4t+u$ with $0\leq u\leq 3$.

{\bf Case 1:} $n\equiv 0\Md{4}$ and $b$ is odd. Then
\begin{align*}
&\sum_{z\in\Z} \frac{g_{n-2b-2,2z-b+1}}{2}[f_i^2S((n-2z-3,2z+1))]\\
&\equiv\sum_{z\in\Z} g_{n-2b-2,2z-b+1}([S((n-2z-1,2z+1))]+2[S((n-2z-2,2z+2))]\\
&\hspace{11pt}+[S((n-2z-3,2z+3))])\Md{\Tspin_n}.
\end{align*}
So
\[[S((n-a,a),\eps):D^{\dbl(n-b,b)}]=\left\{\begin{array}{ll}
2g_{n-2b-2,a-b-1}&a\text{ is even},\\
g_{n-2b-2,a-b}+g_{n-2b-2,a-b-2}&a\text{ is odd}.
\end{array}\right.\]
We have to show that this is equal to $g_{n-2b,a-b}+2g_{n-2b-2,a-b-1}$.

Write $n-2b=4r+2$ and $a-b=4s+t$ with $0\leq t\leq 3$. If $a$ is even then $g_{n-2b,a-b}=0$, so the theorem holds. If $a$ is odd then $g_{n-2b-2,a-b-1}=0$. Further $g_{n-2b,a-b}=g_{4r+2,4s+t}=g_{4r,4s}$. If $t=0$ then
\[g_{n-2b-2,a-b}+g_{n-2b-2,a-b-2}=g_{4r,4s}+g_{4r,4(s-1)+2}=g_{4r,4s},\]
while if $t=2$ then
\[g_{n-2b-2,a-b}+g_{n-2b-2,a-b-2}=g_{4r,4s+2}+g_{4r,4s}=g_{4r,4s}.\]

{\bf Case 2:} $n\equiv 1\Md{4}$ and $b$ is even. Then
\begin{align*}
&\sum_{z\in\Z} \frac{g_{n-2b-1,2z-b+2}}{6}[f_i^3S((n-2z-4,2z+1))]\\
&\equiv\sum_{z\in\Z} g_{n-2b-1,2z-b+2}([S((n-2z-2,2z+2))]+[S((n-2z-3,2z+3))])\\
&\hspace{11pt}\Md{\Tspin_n}.
\end{align*}
So
\[[S((n-a,a),\eps):D^{\dbl(n-b,b)}]=\left\{\begin{array}{ll}
g_{n-2b-1,a-b}&a\text{ is even},\\
g_{n-2b-1,a-b-1}&a\text{ is odd}.
\end{array}\right.\]
We have to show that this is equal to $g_{n-2b,a-b}$.

Write $n-2b=4r+1$ and $a-b=4s+t$ with $0\leq t\leq 3$. If $a$ is even then $t=0$ or 2, so $g_{n-2b-1,a-b}=g_{4r,4s+t}=g_{4r+1,4s+t}=g_{n-2b,a-b}$. If $a$ is odd then $t=1$ or 3, so $g_{n-2b-1,a-b-1}=g_{4r,4s+t-1}=g_{4r+1,4s+t}=g_{n-2b,a-b}$.

{\bf Case 3:} $n\equiv 1\Md{4}$ and $b$ is odd. Then
\begin{align*}
&\sum_{z\in\Z} \frac{g_{n-2b-3,2z-b+1}}{6(2+\de_{i=0})}[e_if_{1-i}f_i^3S((n-2z-4,2z+1))]\\
&\equiv\sum_{z\in\Z} g_{n-2b-3,2z-b+1}([S((n-2z-1,2z+1))]+[S((n-2z-4,2z+4))])\\
&\hspace{11pt}\Md{\Tspin_n}
\end{align*}
(if $i=0$ we can also add and remove a node to the third row). So
\[[S((n-a,a),\eps):D^{\dbl(n-b,b)}]=\left\{\begin{array}{ll}
g_{n-2b-3,a-b-3}&a\text{ is even},\\
g_{n-2b-3,a-b}&a\text{ is odd}.
\end{array}\right.\]
We have to show that this is equal to $g_{n-2b,a-b}-g_{n-2b-2,a-b-1}$.

Write $n-2b=4r+3$ and $a-b=4s+t$ with $0\leq t\leq 3$. If $t=0$ then $g_{n-2b-3,a-b}=g_{4r,4s}$ and
\[g_{n-2b,a-b}-g_{n-2c-2,a-b-1}=g_{4r+3,4s}-g_{4r+1,4(s-1)+3}=g_{4r,4s}+0=g_{4r,4s}.\]
If $t=1$ then $g_{n-2b-3,a-b-3}=g_{4r,4(s-1)+2}=0$ and
\[g_{n-2b,a-b}-g_{n-2c-2,a-b-1}=g_{4r+3,4s+1}-g_{4r+1,4s}=0.\]
If $t=2$ then $g_{n-2b-3,a-b}=g_{4r,4s+2}=0$ and
\[g_{n-2b,a-b}-g_{n-2c-2,a-b-1}=g_{4r+3,4s+2}-g_{4r+1,4s+1}=0.\]
If $t=3$ then $g_{n-2b-3,a-b-3}=g_{4r,4s}$ and
\[g_{n-2b,a-b}-g_{n-2c-2,a-b-1}=g_{4r+3,4s+3}-g_{4r+1,4s+2}=g_{4r,4s}.\]

{\bf Case 4:} $n\equiv 2\Md{4}$ and $b$ is even. Then
\begin{align*}
&\sum_{z\in\Z} \frac{g_{n-2b-2,2z-b+2}}{6}[f_{1-i}f_i^3S((n-2z-5,2z+1))]\\
&\equiv\sum_{z\in\Z} g_{n-2b-2,2z-b+2}(2[S((n-2z-2,2z+2))]+2[S((n-2z-4,2z+4))])\\
&\hspace{11pt}\Md{\Tspin_n}.
\end{align*}
So
\[[S((n-a,a),\eps):D^{\dbl(n-b,b)}]=\left\{\begin{array}{ll}
2g_{n-2b-2,a-b}+2g_{n-2b-2,a-b-2}&a\text{ is even},\\
0&a\text{ is odd}.
\end{array}\right.\]
We have to show that this is equal to $2g_{n-2b,a-b}$.

Write $n-2b=4r+2$ and $a-b=4s+t$ with $0\leq t\leq 3$. If $a$ is odd then $a-b$ is also odd and so $g_{n-2b,a-b}=0$. If $a$ is even then $t=0$ or 2 and we can conclude similarly to the $a$ odd case in Case 1.

{\bf Case 5:} $n\equiv 3\Md{4}$ and $b$ is even. Then
\begin{align*}
&\sum_{z\in\Z} \frac{g_{n-2b-3,2z-b+2}}{12}[f_{1-i}^2f_i^3S((n-2z-6,2z+1))]\\
&\equiv\sum_{z\in\Z} g_{n-2b-3,2z-b+2}([S((n-2z-2,2z+2))]+[S((n-2z-5,2z+5))])\\
&\hspace{11pt}\Md{\Tspin_n}.
\end{align*}
So
\[[S((n-a,a),\eps):D^{\dbl(n-b,b)}]=\left\{\begin{array}{ll}
g_{n-2b-3,a-b}&a\text{ is even},\\
g_{n-2b-3,a-b-3}&a\text{ is odd}.
\end{array}\right.\]
We can show similarly to Case 3 that this equals $g_{n-2b,a-b}-g_{n-2b-2,a-b-1}$.

{\bf Case 6:} $n\equiv 3\Md{4}$ and $b$ is odd. Then
\begin{align*}
&\sum_{z\in\Z} g_{n-2b-1,2z-b+1}[f_iS((n-2z-2,2z+1))]\\
&\equiv\sum_{z\in\Z} g_{n-2b-1,2z-b+1}([S((n-2z-1,2z+1))]+[S((n-2z-2,2z+2))])\\
&\hspace{11pt}\Md{\Tspin_n}.
\end{align*}
So
\[[S((n-a,a),\eps):D^{\dbl(n-b,b)}]=\left\{\begin{array}{ll}
g_{n-2b-1,a-b-1}&a\text{ is even},\\
g_{n-2b-1,a-b}&a\text{ is odd}.
\end{array}\right.\]
It can be proved similarly to Case 2 that this equals $g_{n-2b,a-b}$.
\end{proof}

\section{Proof of Theorem \ref{T170223}}\label{s6}

In this section we will prove Theorem \ref{T170223}.

Let $n\equiv 0\Md{4}$ and $2\leq b\leq\overline{m}_n$ be even. Further let $i=0$ if $b\equiv 0\Md{4}$ and $i=1$ if $b\equiv 2\Md{4}$.

By Lemmas \ref{T100822} and \ref{L150223} and by \cite[Remark 11.2.9]{KBook} we have that
\begin{align*}
[f_if_{1-i}^2f_i^3P^{\dbl(n-b-5,b-1)}:P^{\dbl(n-b,b)}]&=12,\\
[f_iP^{\dbl(n-b,b)}:P^{\dbl(n-b+1,b)}]&=2,\\
[f_iP^{\dbl(n-b,b)}:P^{\dbl(n-b,b+1)}]&=2.
\end{align*}
By Theorem \ref{T100822}
\begin{align*}
[P^{\dbl(n-b-5,b-1)}:S((n-6-c,c))]&=g_{n-2b-4,c-b+1},\\
[P^{\dbl(n-b+1,b)}:S((n-c+1,c))]&=g_{n-2b+1,c-b},\\
[P^{\dbl(n-b,b+1)}:S((n+1-c,c))]&=g_{n-2b-1,c-b-1}-g_{n-2b-3,c-b-2}
\end{align*}
for $0\leq c\leq m_{n-6}$ or $0\leq c\leq m_{n+1}$ respectively.

We will use this to compute upper bounds on the decomposition numbers. In some cases we will also show that these upper bounds actually give the actual decomposition numbers. 
Note that by the above
\begin{align*}
[f_if_{1-i}^2f_i^3P^{\dbl(n-b-5,b-1)}]&\equiv12[P^{\dbl(n-b,b)}]\Md{R_n},\\
[f_iP^{\dbl(n-b,b)}]&\equiv2[P^{\dbl(n-b+1,b)}]+2[P^{\dbl(n-b,b+1)}]\Md{R_{n+1}}.
\end{align*}

Similar to the proof of Theorem \ref{T100822},
\begin{align*}
&[S((n-a,a),\eps):D^{\dbl(n-b,b)}]=[P^{\dbl(n-b,b)}:S((n-a,a))]\\
&\leq\frac{1}{12}[f_if_{1-i}^2f_i^3P^{\dbl(n-b-5,b-1)}:S((n-a,a))]\\
&=\sum_\nu\frac{1}{24}[P^{\dbl(n-b-5,b-1)}:S(\nu)][f_if_{1-i}^2f_i^3S(\nu):S((n-a,a))]\\
&=\sum_{z\in\Z} \frac{g_{n-2b-4,2z-b+2}}{24}[f_if_{1-i}^2f_i^3S((n-2z-7,2z+1)):S((n-a,a))].
\end{align*}
and
\begin{align*}
&[P^{\dbl(n-b+1,b)}:S((n-c+1,c))]+[P^{\dbl(n-b,b+1)}:S((n-c+1,c))]\\
&\leq\frac{1}{24}[f_i^2f_{1-i}^2f_i^3P^{\dbl(n-b-5,b-1)}:S((n-a,a))]\\
&=\sum_\nu\frac{1}{24}[P^{\dbl(n-b-5,b-1)}:S(\nu)][f_i^2f_{1-i}^2f_i^3S(\nu):S((n-a,a))]\\
&=\sum_{z\in\Z} \frac{g_{n-2b-4,2z-b+2}}{24}[f_i^2f_{1-i}^2f_i^3S((n-2z-7,2z+1)):S((n-a,a))].
\end{align*}

Again let $S((c,d)):=0$ whenever $(c,d)$ is not a 2-regular partition. If $2z+1\not\equiv b-1\Md{4}$ then $g_{n-2b-4,2z-b+2}=0$ since $n-2b\equiv 0\Md{4}$. If $n-4z-8=0$ then $2z-b+2=(n-2b-4)/2$ and so $g_{n-2b-4,2z-b+2}=0$. If $2z+1\equiv b-1\Md{4}$ and $n-4z-8\not=0$ then
\begin{align*}
&\frac{1}{12}[f_if_{1-i}^2f_i^3S((n-2z-7,2z+1))]\\
&\equiv2[S((n-2z-2,2z+2))]+2[S((n-2z-3,2z+3))]\\
&\hspace{11pt}+2[S((n-2z-5,2z+5))]+2[S((n-2z-6,2z+6))]\Md{\Tspin_n}
\end{align*}
and
\begin{align*}
&\frac{1}{24}[f_i^2f_{1-i}^2f_i^3S((n-2z-7,2z+1))]\\
&\equiv\de_{z\not=(n-4)/4}[S((n-2z-1,2z+2))]+2[S((n-2z-2,2z+3))]\\
&\hspace{11pt}+(1+\de_{z\not=(n-12)/4})[S((n-2z-5,2z+6))]\\
&\hspace{11pt}+[S((n-2z-6,2z+7))]\Md{\Tspin_{n+1}}.
\end{align*}

It follows that
\begin{align}\label{E170223}
&[S((n-a,a),\eps):D^{\dbl(n-b,b)}]\\
&\leq\frac{1}{12}[f_if_{1-i}^2f_i^3P^{\dbl(n-b-5,b-1)}:S((n-a,a))]\notag\\
&=\left\{\begin{array}{ll}
2g_{n-2b-4,a-b}+2g_{n-2b-4,a-b-4},&a\text{ is even,}\\
2g_{n-2b-4,a-b-1}+2g_{n-2b-4,a-b-3},&a\text{ is odd}
\end{array}\right.\notag
\end{align}
and
\begin{align}\label{E170223_2}
&g_{n-2b+1,c-b}+g_{n-2b-1,c-b-1}-g_{n-2b-3,c-b-2}\\
&=[P^{\dbl(n-b+1,b)}:S((n-c+1,c))]+[P^{\dbl(n-b,b+1)}:S((n-c+1,c))]\notag\\
&\leq\frac{1}{24}[f_i^2f_{1-i}^2f_i^3P^{\dbl(n-b-5,b-1)}:S((n-c+1,c))]\notag\\
&=\left\{\begin{array}{ll}
\de_{c\not=n/2}g_{n-2b-4,c-b}+(1+\de_{c\not=n/2})g_{n-2b-4,c-b-4},&c\text{ is even,}\\
2g_{n-2b-4,c-b-1}+g_{n-2b-4,c-b-5},&c\text{ is odd.}
\end{array}\right.\notag
\end{align}
If equality holds in \eqref{E170223_2} for some $c$ with $c-b\equiv 0$ or $1\Md{4}$, then equality must hold in \eqref{E170223} for $a\in\{c-1,c\}$, since $S((n-c+1,c))$ appears in both $f_iS((n-c+1,c-1))$ and $f_iS((n-c,c))$ (the first one provided $c\geq 1$).

Since $b\leq n/2-4$ we have that $n-2b-4\geq 4$. It can thus be checked (considering all possibilities for $a-b\Md{4}$ and writing all appearing numbers in the form $4r+s$ with $0\leq s\leq 3$) that
\begin{align*}
&2g_{n-2b-3,a-b}+2g_{n-2b-3,a-b-3}\\
&=\left\{\begin{array}{ll}
2g_{n-2b-4,a-b}+2g_{n-2b-4,a-b-4},&a\text{ is even,}\\
2g_{n-2b-4,a-b-1}+2g_{n-2b-4,a-b-3},&a\text{ is odd,}
\end{array}\right.
\end{align*}
so that
\[[S((n-a,a),\eps):D^{\dbl(n-b,b)}]\leq 2g_{n-2b-3,a-b}+2g_{n-2b-3,a-b-3}.\]

If $a-b\equiv 2\Md{4}$ then from $n-2b-3\equiv 1\Md{4}$ (as $n\equiv 0\Md{4}$ and $b$ is even), we have that
\[[S((n-a,a),\eps):D^{\dbl(n-b,b)}]\leq 2g_{n-2b-3,a-b}+2g_{n-2b-3,a-b-3}=0,\]
so that equality holds.

We may now assume that $a-b\not\equiv 2\Md{4}$ and that at least one of $\nu_2((\lfloor (a-b+1)/4\rfloor))\geq\nu_2((n-2b)/4)$ or $g_{n-2b-4,4\lfloor(a-b+1)/4\rfloor-4}=0$ holds. In this case we will show that equality holds in \eqref{E170223_2} with $c=a+1$ if $a-b\equiv 3\Md{4}$, $c=a$ or $a-1$ if $a-b\equiv 0\Md{4}$ or $c=a$ if $a-b\equiv 1\Md{4}$. The assumptions $\nu_2((\lfloor (a-b+1)/4\rfloor))\geq\nu_2((n-2b)/4)$ and $g_{n-2b-4,4\lfloor(a-b+1)/4\rfloor-4}=0$ then become $\nu_2((\lfloor (c-b)/4\rfloor))\geq\nu_2((n-2b)/4)$ and $g_{n-2b-4,4\lfloor(c-b)/4\rfloor-4}=0$ respectively.

If $c=n/2$ and $c-b\equiv0$ or $1\Md{4}$ then $c\equiv 0\Md{4}$, since $b$ and $n/2=c$ are both even. In this case
\[\nu_2((\lfloor (c-b)/4\rfloor))=\nu_2((n-2b)/8)<\nu_2((n-2b)/4).\]
Further the right-hand side of \eqref{E170223_2} is $g_{n-2b-4,c-b-4}=g_{n-2b-4,4\lfloor(c-b)/4\rfloor-4}$. If this is 0 then equality holds, since the right-hand side is non-negative.

We may now assume that $c<n/2$. Write $n-2b=4k$ and $c-b=4\ell+x$ with $x\in\{0,1\}$. If $x=0$ then
\begin{align*}
&g_{n-2b+1,c-b}+g_{n-2b-1,c-b-1}-g_{n-2b-3,c-b-2}\\
&=g_{4k+1,4\ell}+g_{4(k-1)+3,4(\ell-1)+3}-g_{4(k-1)+1,4(\ell-1)+2}\\
&=g_{4k+1,4\ell}+g_{4(k-1)+3,4(\ell-1)+3}
\end{align*}
and
\begin{align*}
g_{n-2b-4,c-b}+2g_{n-2b-4,c-b-4}&=g_{4(k-1),4\ell}+2g_{4(k-1),4(\ell-1)}.
\end{align*}
If $x=1$ then
\begin{align*}
&g_{n-2b+1,c-b}+g_{n-2b-1,c-b-1}-g_{n-2b-3,c-b-2}\\
&=g_{4k+1,4\ell+1}+g_{4(k-1)+3,4\ell}-g_{4(k-1)+1,4(\ell-1)+3}\\
&=g_{4k+1,4\ell+1}+g_{4(k-1)+3,4\ell}
\end{align*}
and
\begin{align*}
2g_{n-2b-4,c-b-1}+g_{n-2b-4,c-b-5}&=2g_{4(k-1),4\ell}+g_{4(k-1),4(\ell-1)}.
\end{align*}
Since again $n-2b-4\geq 4$, it follows that in either of the two cases equality in \eqref{E170223_2} holds if and only if
\begin{align}\label{E210223}
g_{k,\ell}=g_{k-1,\ell}+g_{k-1,\ell-1}.
\end{align}
We will show that this holds whenever
\[\nu_2(\ell)=\nu_2((\lfloor (c-b)/4\rfloor))\geq\nu_2((n-2b)/4)=\nu_2(k)\]
or
\[g_{k-1,\ell-1}=g_{4(k-1),4(\ell-1)}=g_{n-2b-4,4\lfloor(c-b)/4\rfloor-4}=0\]
hold.

If $c=b$ then $\ell=0$. As $n-2b\geq 8$ we have that $k\geq 2$ and then \eqref{E210223} holds.

We may now assume that $\ell>0$ and write $k=2^y(2\overline{k}+1)$ and $\ell=2^z(2\overline{\ell}+1)$ with $\overline{k}$ and $\overline{\ell}$ non-negative integers.

{\bf Case 1:} $\nu_2(\ell)>\nu_2(k)$. Then $z>y$, so that $\ell=2^{y+1}\ell'$ with $\ell'$ integer, and
\begin{align*}
g_{k,\ell}&=g_{2^{y+1}\overline{k}+2^y,2^{y+1}\ell'}=g_{\overline{k},\ell'}\\
g_{k-1,\ell}&=g_{2^{y+1}\overline{k}+2^{y-1}+\ldots+1,2^{y+1}\ell'}=g_{\overline{k},\ell'}\\
g_{k-1,\ell-1}&=g_{2^{y+1}\overline{k}+2^{y-1}+\ldots+1,2^{z+1}\overline{\ell}+2^{z-1}+\ldots+2^y+2^{y-1}+\ldots+1)}=0.
\end{align*}
So \eqref{E210223} holds.

{\bf Case 2:} $\nu_2(\ell)=\nu_2(k)$. In this case $z=y$ so
\begin{align*}
g_{k,\ell}&=g_{2^{y+1}\overline{k}+2^y,2^{y+1}\overline{\ell}+2^y}=g_{\overline{k},\overline{\ell}}\\
g_{k-1,\ell}&=g_{2^{y+1}\overline{k}+2^{y-1}+\ldots+1,2^{y+1}\overline{\ell}+2^y}=0\\
g_{k-1,\ell-1}&=g_{2^{y+1}\overline{k}+2^{y-1}+\ldots+1,2^{y+1}\overline{\ell}+2^{y-1}+\ldots+1}=g_{\overline{k},\overline{\ell}}
\end{align*}
and then \eqref{E210223} holds.

{\bf Case 3:} $\nu_2(\ell)<\nu_2(k)$. We may assume that $g_{k-1,\ell-1}=0$. Since $\nu_2(\ell)<\nu_2(k)$ we have that $g_{k,\ell}=0$. Further $z\leq y-1$, so
\begin{align*}
g_{k-1,\ell}&=g_{2^{y+1}\overline{k}+2^{y-1}+\ldots+1,2^{z+1}\overline{\ell}+2^z}\\
&\leq g_{2^{y+1}\overline{k}+2^{y-1}+\ldots+1,2^{z+1}\overline{\ell}+2^{z-1}+\ldots+1}\\
&=g_{k-1,\ell-1}=0.
\end{align*}
In particular \eqref{E210223} holds also in this case.

\appendix

\section{Examples}\label{appendix}

We show through some small decomposition matrices which parts of the decomposition matrices can be computed using Theorems \ref{t1}, \ref{T100822} and \ref{T170223}. Since if $n\not\equiv 0\Md{4}$ or if $n\leq 8$ and $n\equiv 0\Md{4}$ only the column corresponding to $(\overline{\dbl}(n))^R$ cannot be computed through the first two of these results, we consider only cases $n\geq 12$ and $n\equiv 0\Md{4}$ here and cover the first few such cases.

We color the columns labeling as follows: red if the corresponding column is covered by Theorem \ref{t1}, blue if covered by Theorem \ref{T100822} and green if (partly) covered by Theorem \ref{T170223}.

In general, for the columns of $D^{(\overline{\dbl}(n))^R}$, \cite[Tables III, IV]{Wales} can be used to find the first two entries, but no further information is known (apart for small $n$).

For $n=12$ the decomposition matrix can be recovered from \cite{gap} (and some computations to identify rows and columns), which we use to give the missing decomposition numbers (all in the column of $D^{(7,5)}$). 

We add $=?$ for known decomposition numbers that are not computed using Theorems \ref{t1}, \ref{T100822} and \ref{T170223}

As usual, missing numbers should be interpreted as 0.

\[\begin{tabular}{l|ccccccccc}
&\rotatebox[origin=c]{90}{$\color{red}D^{(12)}\color{black}$}&\rotatebox[origin=c]{90}{$\color{red}D^{(11,1)}\color{black}$}&\rotatebox[origin=c]{90}{$\color{red}D^{(10,2)}\color{black}$}&\rotatebox[origin=c]{90}{$\color{red}D^{(9,3)}\color{black}$}&\rotatebox[origin=c]{90}{$\color{red}D^{(8,4)}\color{black}$}&\rotatebox[origin=c]{90}{$D^{(7,5)}$}&\rotatebox[origin=c]{90}{$\color{blue}D^{(6,5,1)}\color{black}$}&\rotatebox[origin=c]{90}{$\color{green}D^{(6,4,2)}\color{black}$}&\rotatebox[origin=c]{90}{$\color{blue}D^{(5,4,2,1)}\color{black}$}\\ \hline
$S((12),\pm)$&&&&&&1=?\\
$S((11,1),0)$&&&&&&1=?&\tikzmarknode{a}{1}\\
$S((10,2),0)$&&&&&&0=?&2&2\\
$S((9,3),0)$&&&&&&1=?&1&2&1\\
$S((8,4),0)$&\tikzmarknode{c}{\color{white}0\color{black}}&8&&4&&2=?&&&2\\
$S((7,5),0)$&8&8&4&4&\tikzmarknode{d}{2}&1=?&&2&\tikzmarknode{b}{1}
\end{tabular}
\begin{tikzpicture}[overlay,remember picture]
\node[draw, thin, inner sep=1.5pt, fit=(a) (b)] {};
\node[draw, thin, inner sep=1.5pt, fit=(c) (d)] {};
\end{tikzpicture}
\]

\[\begin{tabular}{l|ccccccccccccc}
&\rotatebox[origin=c]{90}{$\color{red}D^{(16)}\color{black}$}&\rotatebox[origin=c]{90}{$\color{red}D^{(15,1)}\color{black}$}&\rotatebox[origin=c]{90}{$\color{red}D^{(14,2)}\color{black}$}&\rotatebox[origin=c]{90}{$\color{red}D^{(13,3)}\color{black}$}&\rotatebox[origin=c]{90}{$\color{red}D^{(12,4)}\color{black}$}&\rotatebox[origin=c]{90}{$\color{red}D^{(11,5)}\color{black}$}&\rotatebox[origin=c]{90}{$\color{red}D^{(10,6)}\color{black}$}&\rotatebox[origin=c]{90}{$D^{(9,7)}$}&\rotatebox[origin=c]{90}{$\color{blue}D^{(8,7,1)}\color{black}$}&\rotatebox[origin=c]{90}{$\color{green}D^{(8,6,2)}\color{black}$}&\rotatebox[origin=c]{90}{$\color{blue}D^{(7,6,2,1)}\color{black}$}&\rotatebox[origin=c]{90}{$\color{green}D^{(7,5,3,1)}\color{black}$}&\rotatebox[origin=c]{90}{$\color{blue}D^{(6,5,3,2)}\color{black}$}\\ \hline
$S((16),\pm)$&&&&&&&&1=?\\
$S((15,1),0)$&&&&&&&&1=?&\tikzmarknode{a2}{1}\\
$S((14,2),0)$&&&&&&&&?&2&2\\
$S((13,3),0)$&&&&&&&&?&1&2&\tikzmarknode{a}{1}\\
$S((12,4),0)$&&&&&&&&?&&&2&2\\
$S((11,5),0)$&&&&&&&&?&1&2&1&2&1\\
$S((10,6),0)$&\tikzmarknode{c2}{\color{white}0\color{black}}&16&\tikzmarknode{c}{\color{white}0\color{black}}&8&&4&&?&2&2&&&2\\
$S((9,7),0)$&8&16&8&8&4&4&\tikzmarknode{d}{2}&?&1&&&2&\tikzmarknode{b}{1}
\end{tabular}
\begin{tikzpicture}[overlay,remember picture]
\node[draw, thin, inner sep=1.5pt, fit=(a) (b)] {};
\node[draw, thin, inner sep=1.5pt, fit=(c) (d)] {};
\node[draw=gray, thin, inner sep=4pt, fit=(a2) (b)] {};
\node[draw=gray, thin, inner sep=4pt, fit=(c2) (d)] {};
\end{tikzpicture}\]


\[{\tiny\begin{tabular}{l|ccccccccccccccccc}
&\rotatebox[origin=c]{90}{$\color{red}D^{(20)}\color{black}$}&\rotatebox[origin=c]{90}{$\color{red}D^{(19,1)}\color{black}$}&\rotatebox[origin=c]{90}{$\color{red}D^{(18,2)}\color{black}$}&\rotatebox[origin=c]{90}{$\color{red}D^{(17,3)}\color{black}$}&\rotatebox[origin=c]{90}{$\color{red}D^{(16,4)}\color{black}$}&\rotatebox[origin=c]{90}{$\color{red}D^{(15,5)}\color{black}$}&\rotatebox[origin=c]{90}{$\color{red}D^{(14,6)}\color{black}$}&\rotatebox[origin=c]{90}{$\color{red}D^{(13,7)}\color{black}$}&\rotatebox[origin=c]{90}{$\color{red}D^{(12,8)}\color{black}$}&\rotatebox[origin=c]{90}{$D^{(11,9)}$}&\rotatebox[origin=c]{90}{$\color{blue}D^{(10,9,1)}\color{black}$}&\rotatebox[origin=c]{90}{$\color{green}D^{(10,8,2)}\color{black}$}&\rotatebox[origin=c]{90}{$\color{blue}D^{(9,8,2,1)}\color{black}$}&\rotatebox[origin=c]{90}{$\color{green}D^{(9,7,3,1)}\color{black}$}&\rotatebox[origin=c]{90}{$\color{blue}D^{(8,7,3,2)}\color{black}$}&\rotatebox[origin=c]{90}{$\color{green}D^{(8,6,4,2)}\color{black}$}&\rotatebox[origin=c]{90}{$\color{blue}D^{(7,6,4,3)}\color{black}$}\\ \hline
$S((20),\pm)$&&&&&&&&&&1=?\\
$S((19,1),0)$&&&&&&&&&&1=?&1\\
$S((19,2),0)$&&&&&&&&&&?&2&2\\
$S((17,3),0)$&&&&&&&&&&?&1&2&\tikzmarknode{a2}{1}\\
$S((16,4),0)$&&&&&&&&&&?&&&2&2\\
$S((15,5),0)$&&&&&&&&&&?&&$\leq2$&1&2&\tikzmarknode{a}{1}\\
$S((14,6),0)$&&&&&&&&&&?&&$\leq4$&&&2&2\\
$S((13,7),0)$&&&&&&&&&&?&&$\leq2$&1&2&1&2&1\\
$S((12,8),0)$&&16&\tikzmarknode{c2}{\color{white}0\color{black}}&16&\tikzmarknode{c}{\color{white}0\color{black}}&8&&4&&?&&&2&2&&&2\\
$S((11,9),0)$&16&16&8&16&8&8&4&4&\tikzmarknode{d}{2}&?&&2&1&&&2&\tikzmarknode{b}{1}
\end{tabular}
\begin{tikzpicture}[overlay,remember picture]
\node[draw, thin, inner sep=1.5pt, fit=(a) (b)] {};
\node[draw, thin, inner sep=1.5pt, fit=(c) (d)] {};
\node[draw=gray, thin, inner sep=4pt, fit=(a2) (b)] {};
\node[draw=gray, thin, inner sep=4pt, fit=(c2) (d)] {};
\end{tikzpicture}}\]


In particular, for $n\leq 23$, the only decomposition numbers which cannot be recovered are $[S((20-a,a),0):D^{(10,8,2)}]$ for $5\leq a\leq 7$ and those in the column of  $D^{(\overline{\dbl}(n))^R}$.

Boxes in the above decomposition matrices point out parts of the different matrices where corresponding decomposition numbers are equal. Similar regions always exist when comparing decomposition matrices for $n$ and $n-4$ (due to formulas for decomposition numbers in the results in the introduction and \cite[Theorem 1.4]{m3}).

\section*{Acknowledgements}

The author thanks the referee for helpful comments.

\end{document}